\documentclass[a4paper,11pt, leqno]{amsart}
\usepackage{amssymb}
\usepackage{amsthm}
\usepackage{mathrsfs, amsfonts, amsmath}
\usepackage{xcolor}
\usepackage{graphicx}
\usepackage{tikz}

\newcommand{\diam}{\operatorname{diam}}
\newcommand{\dist}{\operatorname{dist}}

\newcommand{\R}{{\mathbb R}}

\newcommand{\modu}{\operatorname{mod}}

\newcommand{\hatc}{\hat{\mathbb{C}}}

\newtheorem{theorem}{\textbf{THEOREM}}[section]
\newtheorem{lemma}[theorem]{\textsc{Lemma}}
\newtheorem{proposition}[theorem]{\textsc{Proposition}}
\theoremstyle{definition}
{\theoremstyle{remark} }
\def\charfn_#1{{\raise1.2pt\hbox{$\chi_{\kern-1pt\lower3pt\hbox{{$\scriptstyle#1$}}}$}}}
\def\leq{\leqslant }
\def\geq{\geqslant }

\def\XXint#1#2#3{{\setbox0=\hbox{$#1{#2#3}{\int}$}
\vcenter{\hbox{$#2#3$}}\kern-.5\wd0}}

\begin{document}

\title{Uniformization of planar domains by exhaustion}
\author{Kai Rajala} 
	\address{Department of Mathematics and Statistics, University of Jyv\"askyl\"a, P.O. Box 35 (MaD), FI-40014, University of Jyv\"askyl\"a, Finland.}
	\email{kai.i.rajala@jyu.fi}
\date{}

\thanks{Research supported by the Academy of Finland, project number 308659. 
\newline {\it 2010 Mathematics Subject Classification.} 30C20, 30C35} 
	
\begin{abstract} 
We study the method of finding conformal maps onto circle domains by approximating with finitely connected subdomains. 
Every domain $D \subset \hatc$ admits \emph{exhaustions}, i.e., increasing sequences of finitely connected subdomains $D_j$ whose union is $D$. By Koebe's theorem, 
each $D_j$ admits a conformal map $f_{D_j}$ from $D_j$ onto a circle domain $f_{D_j}(D_j)$. Assuming $f_{D_j} \to f$, our goal is to find out if $f(D)$ is also a circle domain.  

We present a countably connected $D$ with an exhaustion $(D_j)$ so that $(f_{D_j})$ has a limit  
whose image is \emph{not a circle domain}, and a domain $\Omega$ with an exhaustion $(\Omega_j)$ so that $(f_{\Omega_j})$ has a limit whose image has \emph{uncountably many non-point complementary components}. 

On the other hand, we prove that every exhaustion $(D_j)$ of a countably connected $D$ admits a \emph{refinement} so that the image of the corresponding limit map is a circle domain. Our result extends the He-Schramm theorem on the uniformization of countably connected domains and provides a new proof. 
\end{abstract}

\maketitle

\renewcommand{\baselinestretch}{1.2}


\section{Introduction}
\subsection{Background}

The long-standing \emph{Koebe conjecture} \cite{Koe08} predicts that every domain $D \subset \hatc$ admits a conformal map onto a \emph{circle domain}, i.e., 
a domain whose set of complementary components consists of closed disks and points. See \cite{HeSch93} for an overview. Koebe himself proved this to be the case for finitely connected domains, cf. \cite[Theorem 5.1]{Cou50}. Koebe's theorem has been extended to cover finitely connected targets with varying boundary shapes, the most general 
results being those by Brandt \cite{Bra80} and Harrington \cite{Har82}. See \cite{Sch96} for further information.  

A major breakthrough was made by He and Schramm \cite{HeSch93}, who showed that the Koebe conjecture holds for 
countably connected domains. Soon after Schramm \cite{Sch95} introduced the \emph{transboundary extremal length} (or \emph{transboundary modulus}), and applied it to give a simplified proof to the He-Schramm theorem as well as a generalization to uncountably connected ``cofat" domains. See also \cite{HeSch94}, \cite{HeSch95}, \cite{HerKos90}. Recently, results related to the Koebe conjecture have been established in \cite{Bon16}, \cite{HilMos09}, \cite{LuoWu19}, \cite{NtaYou20}, \cite{SolVid20}, and \cite{You16}. 

The proofs by He-Schramm and Schramm apply approximation of a given domain from \emph{outside} by a decreasing sequence of finitely connected domains 
together with Koebe's theorem to construct a sequence of conformal maps whose limit has circle domain image. In this paper, we study a modification of this method 
where a given domain is approximated from \emph{inside} by \emph{exhaustions}, i.e., increasing sequences of finitely connected subdomains. 

Our approach is motivated by the fact that exhaustions offer more flexibility than approximations from outside. They can potentially be applied to gain a better 
understanding of the Koebe conjecture and related problems. The challenge is in finding exhaustions with the desired properties among all the exhaustions 
of a given domain. 

Theorems \ref{mainexample1} and \ref{mainexample2} below show that an arbitrary exhaustion does not work in general; the image of the limit map is not always a circle domain. However, our main result, Theorem \ref{mainthm}, shows that any exhaustion of a countably connected domain admits a \emph{refinement} so that the image of the corresponding limit map is a circle domain. We now describe our results in detail. 


\subsection{Main results} An \emph{exhaustion} $\Phi$ of a domain $D \subset \hatc$ is a sequence of domains $D_j \subset D$, each bounded by finitely many disjoint 
Jordan curves in $D$, such that 
$$D_j \subset D_{j+1} \text{ for all } j=1,2,\ldots \text{ and } D=\cup_j D_j. $$  

We fix disjoint points $a_0,a_1,a_2 \in D_1$. Then by Koebe's theorem there are unique conformal maps $f_j:D_j \to \tilde{D}_j$ onto circle 
domains $\tilde{D}_j \subset \hatc$ so that $f_j(a_k)=a_k$ for $k=0,1,2$. The sequence $(f_j)$ has a subsequence converging locally uniformly to a conformal $f:D \to f(D)$. 
We denote 
$$
\mathcal{F}_\Phi=\{f:D \to f(D): \, f \text{ is the limit of a subsequence of } (f_j)\}. 
$$
If $\mathcal{F}_\Phi$ contains only one map $f$, i.e., if $(f_j)$ converges, we denote $f=f_\Phi$. The use of this notation 
always contains the implicit assumption that $f_ j\to f_\Phi$. 


\begin{theorem} \label{mainexample1}
There is a countably connected domain $D \subset \hatc$ with exhaustion $\Phi$ such that $f_\Phi(D)$ is not a circle domain. 
\end{theorem} 

We denote the set of complementary components of domain $G$ by $\mathcal{C}(G)$. We say that $p \in \mathcal{C}(G)$ is \emph{non-trivial} if $\diam(p)>0$. 

\begin{theorem} \label{mainexample2}
There is a domain $D \subset \hatc$ with exhaustion $\Phi$ such that $\mathcal{C}(f_\Phi(D))$ contains uncountably many non-trivial elements. 
\end{theorem}

Theorems \ref{mainexample1} and \ref{mainexample2} are in sharp contrast to \cite[Theorem 2.1]{Cou50} on \emph{slit domains}, i.e., domains whose sets of complementary components consist of vertical segments and points; if $\Phi$ is an exhaustion of $D$ and if the targets $\tilde{D}_{j}$ above are slit domains so that $f_j \to f$, then $f(D)$ is always a slit domain.  

In view of Theorems \ref{mainexample1} and \ref{mainexample2}, in order to produce a limit map onto a circle domain it is necessary to modify, or refine, a given exhaustion. Let $\Phi=(D_j)$ and $\Phi'=(D'_j)$ be exhaustions of $D$. We say that $\Phi$ is a \emph{refinement} of $\Phi'$, if every $p \in \mathcal{C}(D_j)$ 
is an element of $\mathcal{C}(D'_{j(p)})$ for some $j(p) \geq j$. Our main result reads as follows. 

\begin{theorem} 
\label{mainthm}
Every exhaustion of a countably connected domain $D \subset \hatc$ has a refinement $\Phi$ such that $f_\Phi(D)$ is a circle domain. 
\end{theorem}

Since every domain admits an exhaustion, Theorem \ref{mainthm} gives a new proof to the He-Schramm theorem. Our main tools are transfinite 
induction, which was also used by He-Schramm and Schramm, and Schramm's transboundary modulus.


\section{Proof of Theorem \ref{mainthm}}
Let $G \subset \hatc$ be a domain and $\hat{G}=\hatc / \sim$, where 
$$
x \sim y \text{ if either } x=y \in G \text{ or } x,y \in p \text{ for some } p \in \mathcal{C}(G). 
$$
The corresponding quotient map is $\pi_G:\hatc \to \hat{G}$. Identifying each $x \in G$ and $p \in \mathcal{C}(G)$ with 
$\pi_G(x)$ and $\pi_G(p)$, respectively, we have 
$$
\hat{G}=G \cup \mathcal{C}(G). 
$$ 
A homeomorphism $f:G \to G'$ has a homeomorphic extension $\hat{f}:\hat{G} \to \hat{G'}$. 

Let $\Phi'=(D'_j)$ be an exhaustion of a countably connected domain $D$. We consider the following property: If $q_1 \in \mathcal{C}(D'_{j_1})$ and 
$q_2 \in \mathcal{C}(D'_{j_2})$, $j_1 \geq j_2$, and if $q_1 \cap q_2 \neq \emptyset$, then 
\begin{equation} 
\label{sisalla} 
\text{either } q_1=q_2 \text{ or } q_1 \text{ lies in the interior of } q_2. 
\end{equation} 
It is not difficult to see that any exhaustion $\Phi''$ of $D$ has a refinement $\Phi'$ satisfying \eqref{sisalla}. Since any refinement of $\Phi'$ is also a refinement of $\Phi''$, we conclude that it suffices to prove Theorem \ref{mainthm} for exhaustions satisfying \eqref{sisalla}. 

We prove Theorem \ref{mainthm} using transfinite induction (cf. \cite{ChaKei73}) and the following result. In this paper, \emph{we allow closed disks to have zero diameter}. For instance, in the following proposition a disk $q \in \mathcal{C}(D)$ may be a point component. 

\begin{proposition}
\label{mainprop}
Let $D \subset \hatc$ be a countably connected domain. Fix an exhaustion $\Phi'=(D_j')$ of $D$ satisfying \eqref{sisalla}, $p\in \mathcal{C}(D)$, and an open neighborhood $U$ of $p$ in $\hatc$ such that 
$\overline{U} \in \mathcal{C}(D_n')$ for some index $n$. Moreover, suppose every $f \in \mathcal{F}_{\Phi'}$ satisfies 
\begin{equation} 
\label{upyu}
\hat{f}(q) \text{ is a disk for all } q \in \mathcal{C}(D) \setminus \{p\}, \, q \subset U. 
\end{equation} 
Then $\Phi'$ has a refinement $\Phi_p=(D_j(p))$ such that 
\begin{eqnarray}
\label{aqua1} D_j(p) \setminus U=D_j'\setminus U \quad \text{for all } j \in \mathbb{N} \quad \text{and } \\
\label{aqua2} \text{if } \Phi \text{ is any refinement of } \Phi_p, \text{ then } \hat{g}(p) \text{ is a disk for all } g \in \mathcal{F}_\Phi. 
\end{eqnarray}
\end{proposition}


\subsection{Transfinite induction}
Suppose $D \subsetneq \hatc$ is a countably connected domain. We lose no generality by assuming that the number of complementary components of $D$ is infinite. We denote $E_0=\hat{D} \setminus D$.  For any compact non-empty $E\subset E_0$, let  
$$
E^*=\{p \in E: \, p \, \text{ is not isolated in } E \}. 
$$
By the Baire category theorem, $E^* \subsetneq E$. We can now use transfinite induction to define a well ordered set of subsets $E_\alpha$ of $E_0$ as follows: 
Given an ordinal $\alpha >0$, we define 
\begin{eqnarray*}
E_\alpha=
\left\{
\begin{array}{ll}
(E_\beta)^*, & \text{if }\alpha=\beta+1 \text{ is a successor ordinal, } \\ 
\cap_{\beta < \alpha} E_\beta, & \text{if }\alpha \text{ is a limit ordinal. }  
\end{array}
\right. 
\end{eqnarray*}

It follows that each $E_\alpha$ is compact and $E_\alpha \subsetneq E_\beta$ if $\alpha > \beta$ and $E_\beta \neq \emptyset$. There is an $\alpha_L$ so that $E_{\alpha_{L}}$ is finite and non-empty, thus $E_{\alpha_{L}+1}=\emptyset$. 

We now show how Theorem \ref{mainthm} follows from Proposition \ref{mainprop}. 

\begin{proposition}
\label{transprop} 
Let $\Phi(0)=(D_j(0))$ be an exhaustion of $D$ satisfying \eqref{sisalla}. For every ordinal $0\leq \alpha \leq \alpha_L+1$ there is an exhaustion $\Phi(\alpha)=(D_j(\alpha))$ of $D$ 
so that  
\begin{itemize} 
\item[(i)] if $0 \leq \beta \leq \alpha$, then $\Phi(\alpha)$ is a refinement of $\Phi(\beta)$, and 
\item[(ii)] if $\Phi$ is any refinement of $\Phi(\alpha)$, then 
\begin{equation} 
\label{tura}
\hat{f}(q) \text{ is a disk for all }  f \in \mathcal{F}_\Phi  \text{ and } q \in E_0 \setminus E_{\alpha}. 
\end{equation} 
\end{itemize}
\end{proposition} 

Suppose $\Phi(\alpha_L+1)=(D_j)$ satisfies \eqref{tura}, and let $(f_{j_k})$ be a subsequence of the corresponding $(f_j)$ converging to some $f$. Choosing 
$\Phi=(D_{j_k})$ and $f=f_{\Phi}$ shows that Theorem \ref{mainthm} follows from Proposition \ref{transprop}. 

\begin{proof}[Proof of Proposition \ref{transprop} assuming Proposition \ref{mainprop}]
First, we enumerate the elements $p=p(k)\in \mathcal{C}(D)$, and denote $p(k)\prec p(\ell)$ if $k<\ell$. This should not be confused with the ordering of the sets $E_\alpha$. 
Each $p$ belongs to $E_\alpha \setminus E_{\alpha+1}$ for exactly one $0 \leq \alpha \leq \alpha_L$. Fix such an $\alpha$. Then each 
$p \in E_\alpha \setminus E_{\alpha+1}$ admits an open neighborhood $U_p \subset \hatc$ so that 
$\overline{U}_p \in \mathcal{C}(D_j(0))$ for some $j$, 
\begin{eqnarray}
\label{rs}
& &\pi_D(\overline{U}_p) \cap E_{\alpha+1} = \emptyset, \quad \text{and } \\ 
\label{ss} 
& &\overline{U}_p \cap \overline{U}_q =\emptyset \quad \text{if } q \in E_\alpha \setminus(E_{\alpha+1} \cup \{p\}) \text{ or if } q \in E_0 \setminus E_\alpha \text{ satisfies } q \prec p. 
\end{eqnarray} 

We apply transfinite induction. The claims of the proposition clearly hold for $\alpha=0$ with the given exhaustion $\Phi(0)$. We assume that the claims hold for all $\beta < \alpha$ and verify them for $\alpha$. 

Let $\alpha=\beta+1$ be a successor ordinal. By the induction assumption, \eqref{rs} and \eqref{ss}, Condition \eqref{upyu} in Proposition \ref{mainprop} 
is satisfied with $\Phi'=\Phi(\beta)$, $p \in E_0 \setminus E_\beta$, and $U=U_p$. The proposition combined with our choice of $U_p$ then gives a refinement 
$\Phi(\alpha)=(D_j(\alpha))$ of 
$\Phi(\beta)=(D_j(\beta))$ so that 
\begin{equation} 
\label{amyri}
D_j(\alpha) \setminus \bigcup_{p \in E_\beta \setminus E_\alpha} U_p = D_j(\beta) \setminus \bigcup_{p \in E_\beta \setminus E_\alpha} U_p
\end{equation} 
and so that \eqref{tura} holds for all $p \in E_\beta \setminus E_\alpha$. Notice again that if $\Phi'$ is a refinement of $\Phi$ and if $\Phi''$ is a refinement of $\Phi'$, then $\Phi''$ is a refinement of $\Phi$. 
The claims follow. 

Now let $\alpha=\cap_{\beta<\alpha} \beta$ be a limit ordinal. We define $\Phi(\alpha)=(D_j(\alpha))$ as follows: first, let  
\begin{equation}
\label{pearl} 
D_j(\alpha) \setminus \Big(\bigcup_{p \in E_0 \setminus E_\alpha}  U_p \Big) =D_j(0) \setminus \Big(\bigcup_{p \in E_0 \setminus E_\alpha}  U_p \Big). 
\end{equation} 
Fix $p \in E_0 \setminus E_\alpha$. Each $q \in E_0$ belongs to some $E_{\beta(q)} \setminus E_{\beta(q)+1}$. With this notation, we have $\beta(p) < \alpha$. 

By \eqref{ss} there are only finitely many $q \in E_{\beta(p)} \setminus E_\alpha$ such that 
\begin{equation}
\label{jerki}
\overline{U}_p \cap \overline{U}_q \neq \emptyset. 
\end{equation} 
Moreover, since each such $\overline{U}_q$ belongs to $\mathcal{C}(D_j(0))$, \eqref{rs} and \eqref{ss} show that  
\begin{equation} 
\label{freku} 
U_p \subset U_q \subset U_{q'} 
\end{equation}
if both $\overline{U}_q$ and $\overline{U}_{q'}$ satisfy \eqref{jerki} and $\beta(q) \leq \beta(q')$. 

Among the elements $q$ for which \eqref{jerki} holds, let $q(p)$ be the one with the maximal $\beta(q)$. Then $\beta(p) \leq \beta(q(p))< \alpha$. We set 
\begin{equation} 
\label{myssy}
D_j(\alpha) \cap U_p= D_j(\beta(q(p))) \cap U_p, \quad  p \in E_0 \setminus E_\alpha. 
\end{equation}
Then \eqref{pearl} and \eqref{myssy} define $\Phi(\alpha)=(D_j(\alpha))$. Furthermore, \eqref{amyri}, \eqref{freku}, and the induction assumption show that $\Phi(\alpha)$ is a refinement of every $\Phi(\beta)$, $\beta \leq \alpha$, and that \eqref{tura} holds. The proof is complete, modulo Proposition \ref{mainprop}. 
\end{proof} 


\subsection{Transboundary modulus}
We will apply the following generalization of conformal modulus, first introduced by Schramm \cite{Sch95}. In addition to its importance in classical 
uniformization problems, this method has played a central role in recent developments on the uniformization of fractal metric spaces, cf. 
\cite{Bon11}, \cite{BonMer13}, \cite{BonMer20}, \cite{HakLi19}, \cite{Nta20}.

Let $G \subset \hatc$ be a domain. The \emph{transboundary modulus} $\modu(\Gamma)$ of a family $\Gamma$ of paths in $\hat{G}$ is 
$$
\modu(\Gamma)=\inf_{\rho \in X(\Gamma)} \int_{G} \rho^2 \, dA + \sum_{p \in \mathcal{C}(G)} \rho(p)^2, 
$$
where $X(\Gamma)$ consists of all Borel functions $\rho:\hat{G}\to [0,\infty]$ for which 
$$
1 \leq   \int_{\gamma} \rho \, ds +\sum_{p \in \mathcal{C}(G) \cap |\gamma|} \rho(p)  \quad \text{for all } \gamma \in \Gamma.  
$$
Here $\int_{\gamma} \rho \, ds$ is the path integral of the restriction of $\gamma$ to $G$. More precisely, this restriction is 
a countable union of disjoint paths $\gamma_j$, each of which maps onto a component of $|\gamma| \setminus \mathcal{C}(G)$, and we define 
$$
\int_{\gamma} \rho \, ds = \sum_j \int_{\gamma_j} \rho \, ds.  
$$

As noticed in \cite{Sch95}, the transboundary modulus is a conformal invariant. 

\begin{lemma}
\label{modinvariance} 
Suppose $f:G \to G'$ is conformal. Then for every path family $\Gamma$ and $\hat{f}(\Gamma)=\{\hat{f} \circ \gamma: \, \gamma \in \Gamma\}$ we have 
$$
\modu(\hat{f}(\Gamma)) =\modu(\Gamma). 
$$
\end{lemma}

The proof is a straightforward modification of the proof of the corresponding result for conformal modulus. 

We will prove Proposition \ref{mainprop} by applying the following estimate. Given a domain $G \subset \hatc$ and disjoint sets $A,B \subset \hatc$, we denote 
\begin{eqnarray*}
\Gamma(A,B;G) &=& \{\text{paths in }\hat{G} \text{ joining } \pi_{G}(A) \text{ and }\pi_{G}(B)\}, \\ 
\modu(A,B;G) &=& \modu(\Gamma(A,B;G)).  
\end{eqnarray*}

\begin{proposition}
\label{mainprop2}
Let $D \subset \hatc$ be a countably connected domain. Fix an exhaustion $\Phi'=(D_j')$ of $D$ satisfying \eqref{sisalla}, $p\in \mathcal{C}(D)$, and an open neighborhood $U$ of $p$ in $\hatc$ 
such that $\overline{U} \in \mathcal{C}(D_n')$ for some index $n$. Moreover, suppose every $q \in \mathcal{C}(D) \setminus \{p\}$, $q \subset U$, is a disk. 
Then $\Phi'$ has a refinement $\Phi_p=(D_j(p))$ satisfying 
$$ 
D_j(p) \setminus U=D_j'\setminus U \quad \text{for all } j \in \mathbb{N}
$$ 
so that if $\Phi=(D_j)$ is any refinement of $\Phi_p$, then 
\begin{equation}
\label{anna} 
 \lim_{r \to 0} \limsup_{j \to \infty} \modu(S(z,r)\setminus p_j,\partial U;D_j) =0  
\end{equation}
for every $z \in p$, where $p_j$ is the element of $\mathcal{C}(D_j)$ containing $p$. 
\end{proposition}

We postpone the proof of Proposition \ref{mainprop2} until Section \ref{pro2sec}. We next show that Proposition \ref{mainprop} follows 
from Proposition \ref{mainprop2}. 


\subsection{From Proposition \ref{mainprop2} to Proposition \ref{mainprop}} \label{prosec}
Fix $\Phi'=(D_j')$, $p$, and $U$ as in Proposition \ref{mainprop}. Replacing $D$ with $f(D)$ and $\Phi'$ with $(f(D_j'))$ for some $f \in \mathcal{F}_{\Phi'}$ if necessary, 
we may assume that the assumptions of Proposition \ref{mainprop2} are valid. It then suffices to show that \eqref{anna} implies \eqref{aqua2} in Proposition \ref{mainprop}: if $\Phi$ is any refinement of $\Phi_p$, then $\hat{g}(p)$ is a disk for all $g \in \mathcal{F}_\Phi$. 

Fix a refinement $\Phi=(D_j)$ of $\Phi_p$. As before, let $f_j:D_j \to \tilde{D}_j$ be the associated conformal maps onto circle domains $\tilde{D}_j$. 
Fix $g \in \mathcal{F}_{\Phi}$. By passing to a subsequence if necessary, we may assume that $f_j \to g$. 

Taking another subsequence if necessary, we may assume that $(\hat{f}_j(p_j))$ Hausdorff converges to a closed disk $B$, where $p_j$ is the element 
of $\mathcal{C}(D_j)$ containing $p$ (recall that $B$ may have zero radius). 

Since $f_j \to g$, we have $B \subset \hat{g}(p)$. We will prove that in fact $B=\hat{g}(p)$. This implies \eqref{aqua2}. 

Applying suitable M\"obius transformations if necessary, we may assume that $U,f_j(\partial{U}),g(\partial{U})$ are all subsets 
of the unit disk in the complex plane. Towards contradiction, suppose that $B \subsetneq \hat{g}(p)$. Then 
$$
\dist(w_0,B) \geq 2 \delta \quad \text{for some }w_0 \in \partial \hat{g}(p) \text{ and }\delta >0, 
$$ 
where $\dist$ is the euclidean distance. It follows that there are sequences $(j_k)$ and $(z_k)$ such that $j_k > k$ and $z_k \in \partial p_{k} \subset D_{j_k}$ for all $k \in \mathbb{N}$, and 
\begin{eqnarray*} 
\dist(f_j(z_k),\hat{f}_j(p_j)) \geq \delta \quad \text{for all } j \geq j_k. 
\end{eqnarray*}
By passing to another subsequence if necessary, we may assume that 
$$
z_k \to z \in p. 
$$

Fix $k \in \mathbb{N}$ and $j \geq j_k$. We construct a suitable path family $\Gamma(j,k)$ and estimate its modulus to arrive at a contradiction. 
Let $w \in \mathbb{C}$ be the point in $\hat{f}_j(p_j)$ closest to $f_j(z_k)$, and denote  
$$
I = (w,f_j(z_k)), \quad \ell  = \text{ the line containing segment } I. 
$$ 
Let $V_j$ be the bounded component of $\mathbb{C} \setminus f_j(\partial U)$, and denote the $f_j(z_k)$- and $w$-components of 
$\overline{V}_j \cap (\ell \setminus I)$ by $P'$ and $Q'$, respectively. Moreover, let 
$$
P=  \hat{f}_j^{-1}(\pi_{f_j(D_j)}(P')), \, Q= \hat{f}_j^{-1}(\pi_{f_j(D_j)}(Q'))  \subset \hat{D}_j. 
$$ 

There are unique points $a,b \in \partial U$ so that $\pi_{D_j}(a) \in P$ and $\pi_{D_j}(b) \in Q$. Let $J_1,J_2$ be the connected components of $\partial U \setminus \{a,b\}$, and let 
$$
\Gamma(j,k)=\{ \text{paths joining } \pi_{D_j}(J_1) \text{ and } \pi_{D_j}(J_2) \text{ in } \pi_{D_j}(U) \setminus (P\cup Q) \}. 
$$
Then every $\gamma \in \Gamma(j,k)$ passes through $\pi_j(B(z,|z-z_k|))$, so if we denote $r_k=|z-z_k|$ and choose $k$ large enough so that $S(z,r_k) \subset U$, we have 
$$
\Gamma(j,k) \subset \Gamma(S(z,r_k)\setminus p_j,\partial U;D_j) 
$$
(observe that $\pi_{D_j}(p_j) \in Q$). Thus, 
$$
 \modu(\Gamma(j,k)) \leq \modu(S(z,r_k) \setminus p_j,\partial U;D_j) 
$$
so by \eqref{anna},  
\begin{equation}
\label{haara}
\lim_{k \to \infty} \limsup_{j \to \infty} \modu(\Gamma(j,k)) =0. 
\end{equation}

\begin{lemma} \label{sis}
We have 
\begin{equation} 
\label{polku}
\modu (\hat{f}_j\Gamma(j,k)) \geq M>0 \quad \text{for all } k \in \mathbb{N} \text{ and } j \geq j_k, 
\end{equation} 
where $\hat{f}_j\Gamma(j,k)=\{ \hat{f}_j \circ \gamma: \, \gamma \in \Gamma(j,k)\}$ and $M$ does not depend on $j$ or $k$. 
\end{lemma}

Combining \eqref{haara} with Lemmas \ref{modinvariance} and \ref{sis} leads to a contradiction, so once Lemma \ref{sis} has 
been proved we know that Proposition \ref{mainprop} follows from Proposition \ref{mainprop2}. 

\begin{proof}[Proof of Lemma \ref{sis}]
We consider the subfamily $\Gamma$ of $\hat{f}_j\Gamma(j,k)$ consisting of projections of segments orthogonal to $\ell$. More precisely, 
denote by $T$ the length of $I$, $T=|w-f_j(z_k)|$, and let $\eta(s)=(1-\frac{s}{T})w+\frac{s}{T}f_j(z_k)$, $0<s<T$, be an arc-length parametrization of $I$. Notice that $T \geq \delta$. 

Fix $0<s<T$, and denote by $\ell_s$ the line orthogonal to $\ell$ passing through $\eta(s)$. Then there is a component 
$I_s$ of $\ell_s \cap \overline{V}_j$ with endpoints $m_1 \in f_j(J_1)$ and $m_2 \in f_j(J_2)$ (recall that $V_j$ is the bounded component of $\mathbb{C} \setminus f_j(\partial U)$). Choose a parametrization 
 $\gamma_s$ of $\pi_{D_j}(I_s)$, and let 
$$
\Gamma=\{\gamma_s: \, 0<s<T\}.  
$$ 
Then $\Gamma \subset \hat{f}_j\Gamma(j,k)$, so it suffices to prove \eqref{polku} with $\hat{f}_j\Gamma(j,k)$ replaced by $\Gamma$. 

Fix $\rho \in X(\Gamma)$, and denote by $\mathcal{D}_j$ the family of disks $\tau \in \hat{f}_j(\mathcal{C}(D_j))$ satisfying $\tau \subset \pi_{D_j}(V_j)$. Then  
$$
1 \leq \int_{I_s} \rho \, ds + \sum_{q \in \mathcal{D}_j\cap |\gamma_s|} \rho(q)  \quad \text{ for all } 0<s<T. 
$$ 
Integrating from $0$ to $T$ and applying Fubini's theorem and H\"older's inequality yields 
\begin{eqnarray*}
\delta &\leq& T \leq \int_{f_j(U\cap D_j)} \rho \, dA + \sum_{\tau \in \mathcal{D}_j} \diam(\tau) \rho(\tau) \\
&\leq& |V_j|^{1/2} \Big(\int_{f_j(U\cap D_j)} \rho^2 \, dA \Big)^{1/2} +
\Big(\sum_{\tau \in \mathcal{D}_j} \diam(\tau)^2 \Big)^{1/2} \Big(\sum_{\tau \in \mathcal{D}_j} \rho(\tau)^2 \Big)^{1/2} \\ 
&\leq& \Big( \frac{4}{\pi}  |V_j| \Big)^{1/2}\Big(\Big(\int_{f_j(U\cap D_j)} \rho^2 \, dA \Big)^{1/2} + \Big(\sum_{\tau \in \mathcal{D}_j} \rho(\tau)^2 \Big)^{1/2} \Big),  
\end{eqnarray*}
where the last inequality follows since the disks $\tau$ are disjoint subsets of $V_j$. 

The uniform convergence of $(f_j)$ guarantees that there is 
$N>0$ independent of $j$ such that $\diam(V_j) \leq N$ for all $j$. Combining with the estimate above leads to   
$$
\frac{\pi \delta^2}{16N^2} \leq  \int_{f_j(U\cap D_j)} \rho^2 \, dA + \sum_{\tau \in \mathcal{D}_j} \rho(\tau)^2. 
$$
Since this holds for all $\rho \in X(\Gamma)$, we have $\modu(\Gamma)\geq \pi \delta^2/(16N^2)$. 
\end{proof} 


\subsection{Proof of Proposition \ref{mainprop2}} 
\label{pro2sec}
We use the following notation: if $G,V \subset \hatc$ are domains, then  
$$
\mathcal{C}(G,V)=\{q \in \mathcal{C}(G): \, q \subset V\}.  
$$
\begin{lemma} 
\label{masterlemma1}
Suppose $D$, $\Phi'$, $p$ and $U$ are as in Proposition \ref{mainprop2}. Then $\Phi'$ has a refinement $\Phi_p=(D_j(p))$ so that 
$D_j(p) \setminus U=D_j'\setminus U$ and 
$$
\mathcal{C}(D_j(p),U) = \hat{\mathcal{C}}_{e,j}\cup \hat{\mathcal{C}}_{d,j} \cup \{\hat{p}_j\} 
$$
for all $j \in \mathbb{N}$, where $\hat{p}_j \supset p$ and $\hat{p}_j \notin \hat{\mathcal{C}}_{e,j}\cup \hat{\mathcal{C}}_{d,j}$, 
\begin{equation} 
\label{master1}
\sum_{\hat{q}(j) \in \hat{\mathcal{C}}_{d,j}} \diam(\hat{q}(j))\leq 2^{-j-1}, 
\end{equation} 
and for every $\hat{q}(j) \in \hat{\mathcal{C}}_{e,j}$  there is $q=\overline{B}(x,t) \in \mathcal{C}(D,U)$, $t>0$, such that 
\begin{equation}
\label{master2}
\overline{B}(x,t) \subset \hat{q}(j) \subset B(x,t+s), \quad s=\min\left\{\frac{t}{100}, \frac{\dist(\hat{q}(j),p)}{100}\right\}. 
\end{equation}
\end{lemma} 

\begin{proof}
We have $\mathcal{C}(D,U) = \mathcal{C}_e \cup \mathcal{C}_d \cup \{p\}$, $p \notin  \mathcal{C}_e \cup \mathcal{C}_d$, where $\mathcal{C}_e$ is a family of disks with positive radius and $\mathcal{C}_d$ a 
family of point components. We enumerate the elements of $\mathcal{C}_d$: 
$$ 
\mathcal{C}_d = \{q_1,q_2,\ldots\}. 
$$
We define $\Phi_p=(D_j(p))$ as follows: First, let $D_j(p) \setminus U=D_j' \setminus U$, $j \in \mathbb{N}$. 
To describe the sets $D_j(p)\cap U$, assume that $j=1$ or $j\geq 2$, and $D_k(p)$ has been defined for all $k \leq j-1$. 

We denote by $\hat{p}_j$ the element of $\mathcal{C}(D'_j,U)$ containing $p$. We lose no generality by assuming that $\hat{p}_1 \subset U$. Then 
$\mathcal{C}(D_{j},U \setminus \hat{p}_j)$ is non-empty for all $j \geq 1$.

Each $q \in \mathcal{C}(D,U)$ is contained in some $\hat{q}(j)$ such that   
\begin{itemize}
\item[(i)] $\hat{p}(j)=\hat{p}_j$,  
\item[(ii)] $\hat{q}(j) \in \mathcal{C}(D'_{j'},U \setminus \hat{p}_j)$ for some $j' \geq j$, 
\item[(iii)] if $j \geq 2$ then $\hat{q}(j) \subset \hat{q}(j-1)$ for some $\hat{q}(j-1) \in \mathcal{C}(D_{j-1}(p),U)$, 
\item[(iv)] if $q=q_m \in \mathcal{C}_d$, then 
$$
\diam(\hat{q}_m(j)) \leq 2^{-j-m-1}, 
$$
\item[(v)] if $q=\overline{B}(x,t) \in \mathcal{C}_e$, then $\hat{q}(j)$ satisfies \eqref{master2}. 
\end{itemize}

Denote $\mathcal{Q}_j=\{\hat{q}(j) : \, q \in \mathcal{C}(D,U)\}$. If $\hat{q}(j),\hat{q}'(j) \in \mathcal{Q}_j$, then either 
$\hat{q}(j) \cap \hat{q}'(j) =\emptyset$ or one is contained in the other. Thus we can define $D_j(p) \cap U$ as the domain 
for which $\mathcal{C}(D_j(p),U)$ is the set of maximal elements in $\mathcal{Q}_j$. 

Properties (i)--(iii) guarantee that $\{D_j(p)\}$ is a refinement of $\{D_j'\}$. Moreover, every $\hat{q}(j)$ 
satisfies (iv) or (v). We define 
\begin{eqnarray*} 
\hat{\mathcal{C}}_{d,j} &=& \{\hat{q}(j) \in \mathcal{C}(D_j(p),U) \setminus \{\hat{p}_j\}: \, \hat{q}(j) \text{ satisfies (iv)} \}, \\ 
\hat{\mathcal{C}}_{e,j} &=& \{\hat{q}(j) \in \mathcal{C}(D_j(p),U) \setminus \{\hat{p}_j\}: \, \hat{q}(j) \text{ satisfies (v)} \}. 
\end{eqnarray*}
\end{proof} 


We complete the proof of Proposition \ref{mainprop2} by showing that any refinement $\Phi=(D_j)$ of the $\Phi_p$ in Lemma \ref{masterlemma1} satisfies the remaining estimate 
\eqref{anna}, i.e., 
$$
 \lim_{r \to 0} \limsup_{j \to \infty} \modu(S(z,r)\setminus p_j,\partial U;D_j) =0  \quad \text{for every } z \in p.
$$ 

\begin{lemma} 
\label{masterlemma2} 
Every refinement $\Phi=(D_j)$ of $\Phi_p$ satisfies \eqref{anna}. 
\end{lemma}

\begin{proof} 
Fix $z \in p$ and let $v$ be the largest integer such that 
$$
B(z,e^v)=B(z,R) \subset U. 
$$ 
It suffices to show that if $j$ is large enough, then 
\begin{equation} 
\label{jouru}
\modu(S(z,r)\setminus p_j,S(z,R);D_j) \leq \epsilon(r) \to 0 \quad \text{as } r \to 0, 
\end{equation} 
where $\epsilon(r)$ does not depend on $j$. We will do this by first constructing a suitable sequence of disjoint annuli, 
and then applying them to find admissible functions.

First, let $v_1=v$. Then, fix $n \geq 1$ and assume that $v_n<v_{n-1}<\cdots <v_1$ have been defined. Denote $R_k=e^{v_k}$ and 
$A_k=B(z,R_k) \setminus \overline{B}(z,R_k/e)$, and let $v_{n+1} <v_n$ be the largest integer such that 
$$
\overline{B}(z,R_{n+1}) \cap \overline{B}(x,t+s) = \emptyset \quad \text{for all } q=\overline{B}(x,t) \in \mathcal{C}(D,U), \, q \cap A_n \neq \emptyset,   
$$ 
where $s$ is as in \eqref{master2}. 

Recall from Lemma \ref{masterlemma1} that 
$$
\mathcal{C}(D_j(p),U) = \hat{\mathcal{C}}_{e,j}\cup \hat{\mathcal{C}}_{d,j} \cup \{\hat{p}_j\}. 
$$
We denote $\hat{\mathcal{C}}_{e,j}=\hat{\mathcal{C}}_{b,j} \cup \hat{\mathcal{C}}_{s,j}$, where
\begin{eqnarray*}
\hat{\mathcal{C}}_{b,j} &=& \{\hat{q}(j) \in \hat{\mathcal{C}}_{e,j}: \, \diam(\hat{q}(j)) \geq \dist(\hat{q}(j),z)  \}   , \\
\hat{\mathcal{C}}_{s,j} &=& \{\hat{q}(j) \in \hat{\mathcal{C}}_{e,j}: \, \diam(\hat{q}(j)) < \dist(\hat{q}(j),z)  \}. 
\end{eqnarray*} 
Moreover, let $\mathcal{C}_j=\mathcal{C}_{d,j} \cup \mathcal{C}_{b,j} \cup \mathcal{C}_{s,j}$, where 
\begin{eqnarray*}
\mathcal{C}_{d,j} = \cup_{m \geq j} \hat{\mathcal{C}}_{d,m}, \quad 
\mathcal{C}_{b,j} = \cup_{m \geq j} \hat{\mathcal{C}}_{b,m}, \quad 
\mathcal{C}_{s,j} = \cup_{m \geq j} \hat{\mathcal{C}}_{s,m}. 
\end{eqnarray*}

Fix a refinement $\Phi=(D_j)$ of $\Phi_p$, and $u < v-100$. We denote $r=e^u$. Let $j$ be large enough so that $2^{-j+1} < r/e$, and $p_j$ the element of $\mathcal{C}(D_j,U)$ 
containing $p$. Since $\Phi$ is a refinement of $\Phi_p$, we have $\mathcal{C}(D_j, U \setminus p_j)\subset \mathcal{C}_j$.  
In particular, 
\begin{eqnarray*}  
\mathcal{C}(D_j,U \setminus p_j)  = \mathcal{D}_j \cup \mathcal{B}_j \cup \mathcal{S}_j, \quad \text{where } 
 \mathcal{D}_j \subset \mathcal{C}_{d,j}, \,\, \mathcal{B}_j  \subset \mathcal{C}_{b,j}, \,\, \mathcal{S}_j \subset \mathcal{C}_{s,j}. 
\end{eqnarray*}
By Lemma \ref{masterlemma1} and the definition of the above sets, the following hold: First, 
\begin{equation} 
\label{mila1} 
\sum_{q(j) \in \mathcal{D}_j} \diam(q(j)) \leq 2^{-j} < \frac{r}{2e}. 
\end{equation} 
Secondly, denoting $\mathcal{B}_j(n)=\{q(j) \in \mathcal{B}_j: \, q(j) \cap A_n \neq \emptyset \}$, we have 
$\mathcal{B}_j(n) \cap \mathcal{B}_j(n')=\emptyset$ if $n \neq n'$. Moreover, since every $q(j) \in \mathcal{B}_j(n)$ contains a disk whose area is comparable to the area of $A_n$, the cardinality of $\mathcal{B}_j(n)$ has an absolute bound;  
\begin{equation} 
\label{mila2}
|\mathcal{B}_j(n)| \leq 30 \quad \text{for all } n \in \mathbb{N}. 
\end{equation} 
Finally, every $q(j) \in \mathcal{S}_j$ satisfies 
\begin{equation} 
\label{mila3} 
\diam(q(j))^2 \leq 2\operatorname{Area}(q(j)). 
\end{equation} 
Moreover, denoting $\mathcal{S}_j(n)=\{q(j) \in \mathcal{S}_j: \, q(j) \cap A_n \neq \emptyset \}$, 
we have $\mathcal{S}_j(n) \cap \mathcal{S}_j(n')=\emptyset$ if $n \neq n'$. 

We construct an admissible function 
\begin{equation}
\label{uur}
\rho \in X(\Gamma(S(z,r)\setminus p_j,S(z,R);D_j)) 
\end{equation} 
as follows: let $m$ be the largest integer such that $v_{m+1} \geq u$, and $1 \leq n \leq m$. 
Define $\rho_n:\hat{D}_j \to [0,\infty]$, 
\begin{eqnarray*} 
\rho_n(w)=  \left\{ \begin{array}{ll} 
\frac{1}{m},  & w \in \mathcal{B}_j(n), \\
\frac{2e\diam(w)}{mR_n} & w \in \mathcal{S}_j(n), \\ 
\frac{2}{m |w-z| }, & w \in A_n \cap D_j,  
\end{array} \right. 
\end{eqnarray*}
and $\rho_n(w)=0$ otherwise. We claim that 
\begin{equation}
\label{mosof} 
\frac{1}{m} \leq   \int_{\gamma} \rho_n \, ds +\sum_{q \in \mathcal{C}_j\cap |\gamma|} \rho_n(q)  
\end{equation} 
for all $\gamma \in \Gamma(S(z,r)\setminus p_j,S(z,R);D_j)$. Fix such a $\gamma$, and denote 
\begin{eqnarray*} 
\Omega_1 &=& \{R_n/e < T < R_n: \, T=|y-z| \text{ for some } y \in |\gamma| \cap D_j\}, \\
\Omega_2 &=& \{R_n/e < T < R_n: \, T=|y-z| \text{ for some } y \in w, \, w \in |\gamma| \cap \mathcal{S}_j(n)\}, \\
\Omega_3 &=&  \{R_n/e < T < R_n: \, T=|y-z| \text{ for some } y \in w, \, w \in |\gamma| \cap \mathcal{D}_j(n)\}.  
\end{eqnarray*} 
We may assume that $\gamma$ does not intersect any $w \in \mathcal{B}_j(n)$, otherwise \eqref{mosof} follows directly from the definition of $\rho_n$. 
We then have 
$$
\int_{\Omega_1} \frac{dT}{T}+\int_{\Omega_2} \frac{dT}{T}+\int_{\Omega_3} \frac{dT}{T} \geq 1, 
$$
which combined with \eqref{mila1} yields 
$$
\int_{\Omega_1} \frac{dT}{T}+\int_{\Omega_2} \frac{dT}{T} \geq \frac{1}{2}. 
$$
The definition of $\rho_n$ in $A_n \cap D_j$ yields 
$$
\int_{\gamma} \rho_n  \, ds \geq \frac{2}{m} \int_{\Omega_1} \frac{dT}{T}. 
$$
On the other hand, combining the definition of $\rho_n$ in $\mathcal{S}_j(n)$ with inequality 
$$
\frac{e(\beta-\alpha)}{R_n} \geq \log \beta -\log \alpha, \quad \frac{e}{R_n} \leq \alpha \leq \beta, 
$$ 
yields 
$$
\sum_{q \in \mathcal{S}_j(n) \cap |\gamma| } \rho_n(q) \geq \frac{2}{m} \int_{\Omega_2} \frac{dT}{T}. 
$$
Combining the estimates yields \eqref{mosof}. In particular, $\rho=\sum_{n=1}^m \rho_n$ satisfies \eqref{uur}, i.e., 
$\rho$ is admissible for $\Gamma(S(z,r)\setminus p_j,S(z,R);D_j)$. 

We prove \eqref{jouru} by estimating the energy 
\begin{equation}
\label{hitme}
\int_{D_j \cap U} \rho^2 \, dA + \sum_{w \in \mathcal{C}_j} \rho(w)^2 
\end{equation}
from above. First, we have 
\begin{equation}
\label{contu}
\int_{D_j \cap U} \rho^2 \, dA \leq \frac{4}{m^{2}}\sum_{n=1}^m  \int_{D_j \cap A_n} \frac{dA(w)}{|w-z|^2}  \leq \frac{8\pi}{m}. 
\end{equation}

In order to estimate the sum in \eqref{hitme}, we recall that each $w \in \mathcal{B}_j \cup \mathcal{S}_j$ intersects at most one $A_n$. By \eqref{mila2}, 
\begin{equation}
\label{disku1}
\sum_{w \in \mathcal{B}_j} \rho(w)^2 \leq \sum_{n=1}^m \frac{|\mathcal{B}_j(n)|}{m^2} \leq \frac{30}{m}.  
\end{equation} 
Finally, since every $w \in \mathcal{S}_j(n)$ is a subset of $B(z,2R_n)$, \eqref{mila3} yields 
\begin{eqnarray}
 \label{disku2}
\sum_{w \in \mathcal{S}_j} \rho(w)^2 
&\leq& \frac{4e^2}{m^2} \sum_{n=1}^m  \sum_{w \in \mathcal{S}_j(n)} \frac{\diam(w)^2}{R_n^2}  \\ 
\nonumber &\leq&  \frac{8e^2}{m^2} \sum_{n=1}^m  \frac{\operatorname{Area}(B(z,2R_n))}{R_n^2}= \frac{32\pi e^2}{m}. 
\end{eqnarray} 
Combining \eqref{contu}, \eqref{disku1} and \eqref{disku2}, we conclude 
$$
\int_{D_j \cap U} \rho^2 \, dA + \sum_{w \in \mathcal{C}_j} \rho(w)^2 \leq \frac{1000}{m} \to 0 \quad \text{as } r \to 0, 
$$
and \eqref{jouru} follows. The proof is complete. 
\end{proof} 


\section{Proof of Theorem \ref{mainexample1} }

\subsection{Construction of the domain}
We will construct a countably connected \emph{square domain}\footnote{The construction of $D$ is flexible in terms of the shapes of the complementary components. In particular, there are circle domains $D$ satisfying the requirements of Theorem \ref{mainexample1}. We use squares in our construction for convenience of presentation. } $D\subset \hatc$ so that $\{0\} \in \mathcal{C}(D)$, and an exhaustion $\Phi$ of $D$ so that $\hat{f}_\Phi(\{0\})$ is non-trivial. 
The following result, which follows from the modulus estimate in \cite[Theorem 6.2]{Sch95}, then shows that $f_\Phi(D)$ cannot be a circle domain. 
\begin{proposition} 
\label{terro}
If $f$ is a conformal map from domain $D \subset \hatc$ with the above properties onto a circle domain, then $\hat{f}(\{0\})$ is a point-component. 
\end{proposition}

We start the construction of $D$ with a sequence of disjoint squares 
$$
Q_k=[a_k-R_k,a_k+R_k]\times [-R_k,R_k], \quad R_1=1, \, R_k < a_k, 
$$ 
where $(a_k)_{k=1}^\infty$, $(R_k)_{k=1}^\infty$ are decreasing sequences converging to zero, so that 
\begin{equation} \label{kulta}
D_k:=\dist(Q_k,Q_{k+1}) =a_k-(a_{k+1}+R_k+R_{k+1})= 2^{-k}R_{k+1}. 
\end{equation} 
Each $Q_k$, $1 \leq k \leq j$, is surrounded by a sequence $(Q_{k,j})$ of inflated squares 
$$
Q_{k,j}=[a_k-T_{k,j},a_k+T_{k,j}]\times [-T_{k,j},T_{k,j}], \quad T_{k,j}=R_k+2^{-j-1}D_k. 
$$
We also denote 
$$
Q_{0,j}=[-T_j,T_j]\times [-T_j,T_j], \quad T_j=a_{j+1}+R_{j+1}+D_j/2. 
$$
Then 
$$
\cup_{k=j+1}^\infty Q_{j} \subset \operatorname{int}(Q_{0,j}), \quad \cup_{k=1}^j Q_{k,j} \cap Q_{0,j}= \emptyset \quad \text{for every } j \in \mathbb{N}. 
$$

Next, for $m \in \mathbb{N}$ and $1 \leq \ell \leq M_m$ ($M_m$ will be chosen later), let 
\begin{equation}
\label{juou}
q_{m,\ell} = [(\ell-1)(s_m+d_m),(\ell-1) s_m+\ell d_m] \times [0,d_m], 
\end{equation}
where $d_m,s_m$ are positive numbers so that 
$$
(M_m-1)s_m+M_md_m=1 \quad \text{and } \quad d_m \geq s_m. 
$$ 
In particular, $d_m \leq M_m^{-1}$. For a fixed $m \in \mathbb{N}$, the sets $q_{m,\ell}$ are evenly spaced squares of sidelength $d_m$ inside the rectangle $[0,1]\times [0,d_m]$. 

For each $m \in \mathbb{N}$ and $1 \leq k \leq m$, let $\phi_{k+1,m}$ be the M\"obius transformation so that $\phi_{k+1,m}(\infty)=\infty$, 
\begin{eqnarray*}
\phi_{k+1,m}(0,0) &=& (a_{k+1}+(1-s_m)T_{k+1,m-1},-R_{k+1}) \quad \text{ and } \\ 
\phi_{k+1,m}(1,0) &=& (a_{k+1}+(1-s_m)T_{k+1,m-1},R_{k+1}).  
\end{eqnarray*} 
We denote 
\begin{equation}
\label{kahwa} 
t_{k,m}=(a_k-T_{k,m-1})-(a_{k+1}+(1-2s_m)T_{k+1,m-1})+d_m, 
\end{equation} 
and define 
\begin{eqnarray*}
q^e_{k+1,m,\ell} &=& \phi_{k+1,m}(q_{m,\ell}), \\ 
q^w_{k,m,\ell} &=& \phi_{k+1,m}(q_{m,\ell})+ (t_{k,m},0), \quad 1 \leq \ell \leq M_m. 
\end{eqnarray*}

The squares $q^e_{k+1,m,\ell},q^w_{k,m,\ell}$ lie ``between'' $Q_{k+1}$ and $Q_k$, and we can choose $M_m$ large enough so that 
\begin{eqnarray}
\nonumber q^e_{m+1,m,\ell} &\subset& \operatorname{int}(Q_{0,m}), \\ 
\label{jalka1} q^e_{k+1,m,\ell} &\subset& \operatorname{int}(Q_{k+1,m-1}) \setminus  Q_{k+1,m} \quad \text{for all } 1 \leq k \leq m-1, \\  
\label{jalka2} q^w_{k,m,\ell} &\subset& \operatorname{int}(Q_{k,m-1}) \setminus Q_{k,m} \quad \text{for all } 1 \leq k \leq m. 
\end{eqnarray}

We define $D$ by 
$$
\hatc \setminus D =\{0\} \cup \Big( \bigcup_{k=1}^\infty  Q_k  \cup \Big(\bigcup_{m=1}^\infty \bigcup_{\ell=1}^{M_m}\bigcup_{k=1}^{m} (q^e_{k+1,m,\ell} \cup q^w_{k,m,\ell}) \Big) \Big). 
$$

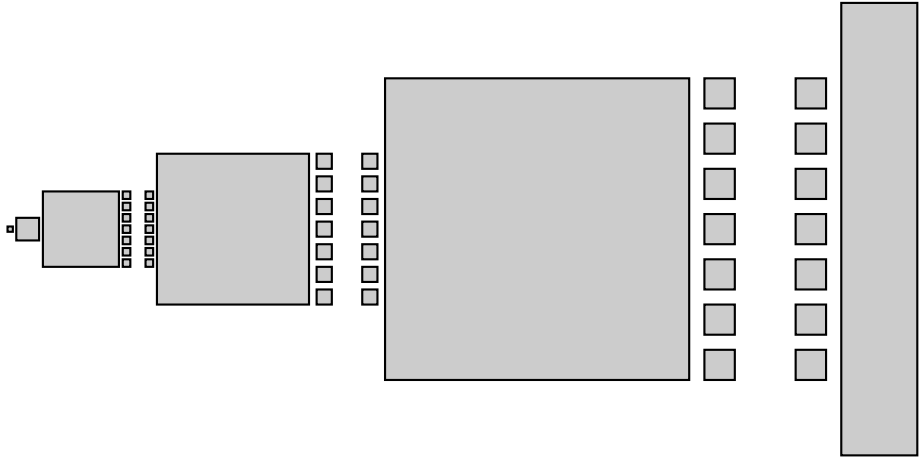
\begin{figure} \label{ekakuva}
\begin{tikzpicture} 
\filldraw[color=black, fill=black!20, thick] (.035,-.035) rectangle (.105,.035);
\filldraw[color=black, fill=black!20, thick] (.15,-.15) rectangle (.45,.15);
\filldraw[color=black, fill=black!20, thick] (.5,-.5) rectangle (1.5,.5); 

\filldraw[color=black, fill=black!20, thick] (1.55,-.5) rectangle (1.65,.-.4);
\filldraw[color=black, fill=black!20, thick] (1.85,-.5) rectangle (1.95,.-.4);
\filldraw[color=black, fill=black!20, thick] (1.55,-.35) rectangle (1.65,.-.25);
\filldraw[color=black, fill=black!20, thick] (1.85,-.35) rectangle (1.95,.-.25);
\filldraw[color=black, fill=black!20, thick] (1.55,-.2) rectangle (1.65,.-.1);
\filldraw[color=black, fill=black!20, thick] (1.85,-.2) rectangle (1.95,.-.1);
\filldraw[color=black, fill=black!20, thick] (1.55,-.05) rectangle (1.65,.05);
\filldraw[color=black, fill=black!20, thick] (1.85,-.05) rectangle (1.95,.05);
\filldraw[color=black, fill=black!20, thick] (1.55,.1) rectangle (1.65,.2);
\filldraw[color=black, fill=black!20, thick] (1.85,.1) rectangle (1.95,.2);
\filldraw[color=black, fill=black!20, thick] (1.55,.25) rectangle (1.65,.35);
\filldraw[color=black, fill=black!20, thick] (1.85,.25) rectangle (1.95,.35);
\filldraw[color=black, fill=black!20, thick] (1.55,.4) rectangle (1.65,.5);
\filldraw[color=black, fill=black!20, thick] (1.85,.4) rectangle (1.95,.5);

\filldraw[color=black, fill=black!20, thick] (2,-1) rectangle (4,1);

\filldraw[color=black, fill=black!20, thick] (4.1,-1) rectangle (4.3,-.8);
\filldraw[color=black, fill=black!20, thick] (4.7,-1) rectangle (4.9,-.8);

\filldraw[color=black, fill=black!20, thick] (4.1,-.7) rectangle (4.3,-.5);
\filldraw[color=black, fill=black!20, thick] (4.7,-.7) rectangle (4.9,-.5);

\filldraw[color=black, fill=black!20, thick] (4.1,-.4) rectangle (4.3,-.2);
\filldraw[color=black, fill=black!20, thick] (4.7,-.4) rectangle (4.9,-.2);

\filldraw[color=black, fill=black!20, thick] (4.1,-.1) rectangle (4.3,.1);
\filldraw[color=black, fill=black!20, thick] (4.7,-.1) rectangle (4.9,.1)

;\filldraw[color=black, fill=black!20, thick] (4.1,.2) rectangle (4.3,.4);
\filldraw[color=black, fill=black!20, thick] (4.7,.2) rectangle (4.9,.4);

\filldraw[color=black, fill=black!20, thick] (4.1,.5) rectangle (4.3,.7);
\filldraw[color=black, fill=black!20, thick] (4.7,.5) rectangle (4.9,.7);

\filldraw[color=black, fill=black!20, thick] (4.1,.8) rectangle (4.3,1);
\filldraw[color=black, fill=black!20, thick] (4.7,.8) rectangle (4.9,1);

\filldraw[color=black, fill=black!20, thick] (5,-2) rectangle (9,2); 

\filldraw[color=black, fill=black!20, thick] (9.2,-2) rectangle (9.6,-1.6);
\filldraw[color=black, fill=black!20, thick] (10.4,-2) rectangle (10.8,-1.6);

\filldraw[color=black, fill=black!20, thick] (9.2,-1.4) rectangle (9.6,-1);
\filldraw[color=black, fill=black!20, thick] (10.4,-1.4) rectangle (10.8,-1);

\filldraw[color=black, fill=black!20, thick] (9.2,-.8) rectangle (9.6,-.4);
\filldraw[color=black, fill=black!20, thick] (10.4,-.8) rectangle (10.8,-.4);

\filldraw[color=black, fill=black!20, thick] (9.2,-.2) rectangle (9.6,.2);
\filldraw[color=black, fill=black!20, thick] (10.4,-.2) rectangle (10.8,.2);

\filldraw[color=black, fill=black!20, thick] (9.2,.4) rectangle (9.6,.8);
\filldraw[color=black, fill=black!20, thick] (10.4,.4) rectangle (10.8,.8);

\filldraw[color=black, fill=black!20, thick] (9.2,1) rectangle (9.6,1.4);
\filldraw[color=black, fill=black!20, thick] (10.4,1) rectangle (10.8,1.4);

\filldraw[color=black, fill=black!20, thick] (9.2,1.6) rectangle (9.6,2);
\filldraw[color=black, fill=black!20, thick] (10.4,1.6) rectangle (10.8,2);

\filldraw[color=black, fill=black!20, thick] (11,-3) rectangle (12,3); 
\end{tikzpicture}
\caption{Part of the complement of $D$} 
\end{figure}

\subsection{Construction of the exhaustion}
We define exhaustion $\Phi_0=(D_j)$ of $D$. First, every $\mathcal{C}(D_j)$ includes 
$$
Q_{0,j} \quad \text{and} \quad Q_{k,j}, \quad 1 \leq k \leq j. 
$$

To describe the rest of the elements of $\mathcal{C}(D_j)$, we first define 
\begin{eqnarray}
\label{cure1} q^e_{k+1,m,\ell,j}  &=& (1+\epsilon(j))q^e_{k+1,m,\ell}  \quad 1 \leq k \leq j-1, \\
\label{cure2} q^w_{k,m,\ell,j}     &=& (1+\epsilon(j))q^w_{k,m,\ell}  \quad 1 \leq k \leq j  
\end{eqnarray} 
for all $k \leq m \leq j$ and $1 \leq \ell \leq M_m$, 
i.e., squares with same center and $(1+\epsilon(j))$ times the sidelength of $q^e_{k+1,m,\ell}$ and $q^w_{k,m,\ell}$, respectively. Here $(\epsilon(j))$ is a strictly decreasing sequence 
converging to zero, and $\epsilon(1)$ is small enough such that for any fixed $j \in \mathbb{N}$ we have 
\begin{itemize} 
\item[(i)] none of the squares intersect each other, and  
\item[(ii)] \eqref{jalka1} holds for $q^e_{k+1,m,\ell,j}$ and \eqref{jalka2} holds for $q^w_{k,m,\ell,j}$. 
\end{itemize}
We let $\mathcal{C}(D_j)$ include all the squares in \eqref{cure1} and \eqref{cure2} for $k \leq m \leq j-1$ and $1 \leq \ell \leq M_m$. Notice that 
the squares for which $m=j$ are not included.  

The remaining elements of $\mathcal{C}(D_j)$ will be components $\overline{U}_{k,j,\ell}$ which ``surround" $Q_{k,j}$ and contain both $q^e_{k,j,\ell,j}$ and $q^w_{k,j,\ell,j}$. More precisely, fix 
$2 \leq k \leq j$, and let $U_{k,j,\ell}$, $1 \leq \ell \leq M_j$, be Jordan domains so that 
\begin{itemize} 
\item[(i)] the sets  $\overline{U}_{k,j,\ell}$ are pairwise disjoint, 
\item[(ii)] $\overline{U}_{k,j,\ell} \subset \operatorname{int}(Q_{k,j-1})\setminus Q_{k,j}$, 
\item[(iii)] $\overline{U}_{k,j,\ell}$ contains both $q^e_{k,j,\ell,j}$ and $q^w_{k,j,\ell,j}$, 
\item[(iv)] if $(x,y) \in \partial U_{k,j,\ell}$ has the largest $x$-coordinate among all points of $\partial U_{k,j,\ell}$, then $(x,y) \in q^e_{k,j,\ell,j}$, 
\item[(v)] if $(x,y) \in \partial U_{k,j,\ell}$ has the smallest $x$-coordinate among all points of $\partial U_{k,j,\ell}$, then $(x,y) \in q^w_{k,j,\ell,j}$. 
\end{itemize}

We conclude the definition of $\mathcal{C}(D_j)$ by including $\overline{U}_{1,j,\ell}:=q^w_{1,j,\ell,j}$ and 
$$
\overline{U}_{k,j,\ell} \quad 2 \leq k \leq j, \, 1 \leq \ell \leq M_j. 
$$
Then $D_j$ is the set for which $\hatc \setminus D_j =  \cup \{p \in \mathcal{C}(D_j)\}$, and $\Phi_0=(D_j)$.

\begin{figure} \label{tokakuva}
\begin{tikzpicture} 
\filldraw[black] (0,0) circle (1pt);

\filldraw[color=blue, fill=blue!20] (1.85,.4) rectangle (1.95,.5); 

\filldraw[color=blue, fill=blue!20]  (1.99,-1.01) rectangle (4.01,1.01);







\filldraw[color=blue, fill=blue!20]  (4.08,.78) rectangle (4.32,1.02); 

\filldraw[color=blue, fill=blue!20] (4.68,.78) rectangle (4.92,1.02); 

\filldraw[color=blue, fill=blue!20]  (4.98,-2.02) rectangle (9.02,2.02); 







\filldraw[color=blue, fill=blue!20]  (9.18,1.58) rectangle (9.62,2.02);
\draw[blue, very thick] (9.4,2.05) -- (9.4,2.3) -- (4.8,2.3) -- (4.8,1.02);
\draw[blue, very thick] (4.2,1.02) -- (4.2, 1.25) -- (1.9,1.25) -- (1.9, 0.5); 

\filldraw[color=blue, fill=blue!20]  (10.38,1.58) rectangle (10.82,2.02);
\draw[blue, very thick] (10.6,2.02) -- (10.6,3.4) -- (12,3.4); 

\filldraw[color=blue, fill=blue!20]  (10.38,.98) rectangle (10.82,1.42); 
\draw[blue, very thick] (10.82,1.2) -- (10.91,1.2) -- (10.91,3.2) -- (12,3.2); 

\filldraw[color=blue, fill=blue!20] (10.98,-3.02) rectangle (12.02,3.02);

\filldraw[color=blue, fill=blue!20]  (-1.7,-1.7) rectangle (1.7,1.7); 

\end{tikzpicture}
\caption{Some of the sets $\overline{U}_{k,j,\ell}$} 
\end{figure}
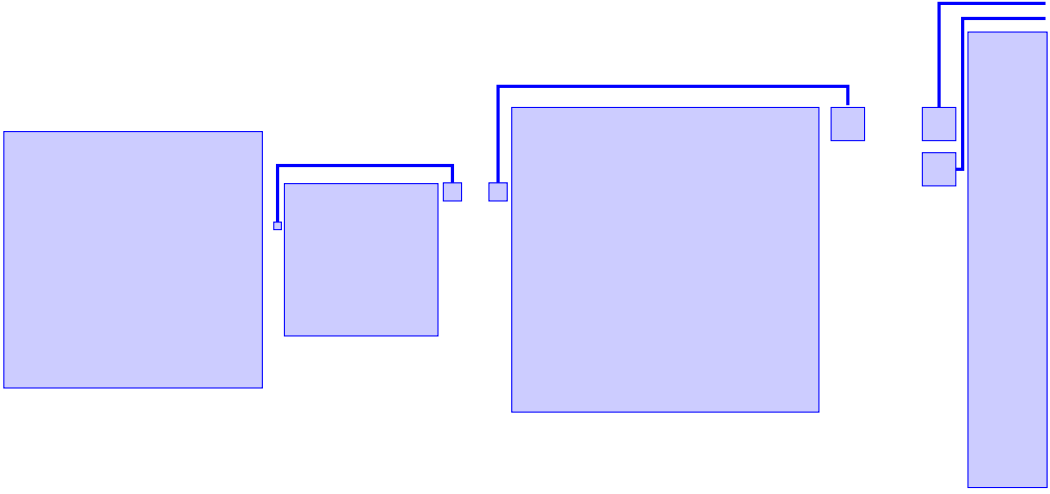

\begin{proposition} 
\label{juba1}
There is $\delta >0$ such that 
$$ 
\modu(Q_{0,j_0},Q_{1,j_0};D_j) \geq \delta \quad \text{for all } j_0 \in \mathbb{N} \text{ and } j \geq j_0. 
$$ 
\end{proposition} 

We postpone the proof of Proposition \ref{juba1} and first show how it implies Theorem \ref{mainexample1}. 
Choose any subsequence $\Phi=(D_{j_n})$ of $(D_j)$ so that $(f_{j_n})$ converges to $f_\Phi$. By 
Proposition \ref{terro} it suffices to show that 
$\hat{f}_\Phi(\{0\})$ is non-trivial. But this follows directly by combining Proposition \ref{juba1} with Proposition \ref{combi} below. 
Here the latter proposition can be applied with $E=Q_{1,1}$ since every $j_0 \geq 1$ satisfies 
$$ 
\modu(Q_{0,j_0},Q_{1,j_0};D_j) \leq \modu(Q_{0,j_0},Q_{1,1};D_j). 
$$
Thus, Theorem \ref{mainexample1} follows once we have proved these propositions.


\subsection{Proof of Proposition \ref{juba1}} 
Fix $j_0$ and $j \geq j_0$, and let $F_j$ be the projection of $\cup_{\ell=1}^{M_j} q_{j,\ell}$ to the real axis, recall the definition in \eqref{juou}. 
We construct a family of paths $\Gamma$ parametrized by $F_j$, 
so that each $\gamma \in \Gamma$ connects $\pi_{D_j}(Q_{1,j_0})$ and $\pi_{D_j}(Q_{0,j_0})$ in $\hat{D}_j$. We then give a lower bound for $\modu(\Gamma)$. 

Fix $\tau \in F_j$, and denote 
\begin{eqnarray*}
z^e_{k+1}(\tau) &=& \phi_{k+1,j}((\tau,d_j/2)), \quad  1 \leq k \leq j-1,  \\ 
z^w_{k}(\tau) &=& \phi_{k+1,j}((\tau,d_j/2))+(t_{k,j},0),  \quad 1 \leq k \leq j, 
\end{eqnarray*}
where $t_{k,j}$ is the number in \eqref{kahwa} and $\phi_{k+1,j}$ the M\"obius transformation defined before \eqref{kahwa}. Then 
\begin{eqnarray*}
z^e_{k+1}(\tau) \in q^e_{k+1,j,\ell,j} \subset \overline{U}_{k+1,j,\ell}  \quad \text{and} \quad 
z^w_{k}(\tau) \in q^w_{k,j,\ell,j} \subset \overline{U}_{k,j,\ell}, 
\end{eqnarray*}
where $\ell=\ell(j,\tau)$ is the index for which $(\tau,0) \in q_{j,\ell}$. 

We denote 
$$
\overline{U}_{k,j,\ell}=:\overline{U}_k(\tau), 
$$ 
and let $I_k(\tau)$ be the horizontal line segment in $\mathbb{C}$ which connects  
$Q_{1,j_0}$ to $z^e_{2}(\tau)$, $z^w_k(\tau)$ to $z^e_{k+1}(\tau)$ if $2 \leq k \leq j-1$, and $z^w_j(\tau)$ to $Q_{0,j}$. 
Then  
$$
J(\tau)=( \cup_{k=1}^j I_k(\tau)) \cup ( \cup_{k=2}^j \overline{U}_k(\tau) )  
$$
is a continuum connecting $Q_{1,j_0}$ and $Q_{0,j_0}$ in $\mathbb{C}$. Moreover, $\pi_{D_j}(J(\tau))$ is a rectifiable curve in $\hat{D}_j$, with 
arc-length parametrization $\gamma_{\tau}$. We define 
$$
\Gamma = \{\gamma_{\tau} : \, \tau \in F_j\}. 
$$ 

We now estimate the modulus of $\Gamma$. Let $\rho \in X(\Gamma)$ be an admissible function and $\tau \in F_j$. 
We denote by $\mathcal{A}_j$ all the sets in $\mathcal{C}(D_j)$ of the form $\overline{U}_{k,j,\ell}$, and 
by $\mathcal{B}_j$ all the other squares in $\mathcal{C}(D_j)$ of the form \eqref{cure1} or \eqref{cure2}. Then 
\begin{eqnarray} 
\label{kiima} 
1 &\leq& \int_{\gamma_{\tau}} \rho \, ds + \sum_{q \in \mathcal{C}(D_j)\cap |\gamma_\tau|} \rho(q)  \\ \nonumber 
&=& \sum_{k=1}^j \int_{I_k(\tau) \cap D_j} \rho \, ds + \sum_{k=1}^j \rho(\overline{U}_k(\tau)) + 
\sum_{q\in \mathcal{B}_j \cap |\gamma_\tau|} \rho(q). 
\end{eqnarray} 

Given $1 \leq k \leq j$, let $A_k$ be the smallest rectangle containing all the segments $I_k(\tau)\cap D_j$, $\tau \in F_j$. Then by \eqref{kulta}, 
\begin{equation} 
\label{sulta}
\operatorname{Area}(A_k) \leq 2D_{k}R_{k+1} \leq 2^{1-k}R_{k+1}^2. 
\end{equation} 

To estimate the modulus, we integrate both sides of \eqref{kiima} over $\tau$ and apply change of variables and Fubini's theorem to get 
\begin{eqnarray} 
\label{vihu}
\ell(F_j) &\leq& \sum_{k=1}^j (2R_{k+1})^{-1}\int_{A_k \cap D_j} \rho \, dA+ \int_{F_j} \sum_{k=1}^j \rho(\overline{U}_k(\tau)) \, d\tau  \\
\nonumber &+& \int_{F_j }\sum_{q\in \mathcal{B}_j \cap |\gamma_\tau|} \rho(q) \, d\tau=S_1+S_2+S_3. 
\end{eqnarray} 

We apply H\"older's inequality and \eqref{sulta} to estimate $S_1$ as follows: 
\begin{eqnarray} 
\label{jytaa} S_1 &\leq&   \sum_{k=1}^j (2R_{k+1})^{-1}\operatorname{Area}(A_k)^{1/2} \Big( \int_{A_k \cap D_j} \rho^2 \, dA \Big)^{1/2}  \\
\nonumber &\leq& \Big( \sum_{k=1}^j 2^{-1-k} \Big)^{1/2}  \Big( \int_{D_j} \rho^2 \, dA  \Big)^{1/2} \leq  \Big( \int_{D_j} \rho^2 \, dA  \Big)^{1/2}.  
\end{eqnarray}

To estimate $S_2$ and $S_3$, we choose $M_m$ so that 
\begin{equation} 
\label{crazy} 
M_m \geq m 2^{m+1}   \quad \text{for all } m \in \mathbb{N}. 
\end{equation}  
We notice that the length of the set of parameters $\tau$ for which a given $\overline{U}_{k,j,\ell} \in \mathcal{A}_j$ is $\overline{U}_k(\tau)$ equals $d_j$. 
We have $d_j M_j \leq 1$ by construction. Thus, H\"older's inequality and \eqref{crazy} yield 
\begin{eqnarray}
\nonumber S_2 &=& d_j \sum_{k=1}^j \sum_{\ell=1}^{M_j} \rho(\overline{U}_{k,j,\ell})
\leq  d_j(jM_j)^{1/2}  \Big( \sum_{k=1}^j \sum_{\ell=1}^{M_j}  \rho(\overline{U}_{k,j,\ell,j})^2 \Big)^{1/2} \\
\label{jytaa2} &\leq& \Big( \sum_{\overline{U} \in \mathcal{A}_j}  \rho(\overline{U})^2 \Big)^{1/2}. 
\end{eqnarray}

Next, we notice that the length of the set of parameters $\tau$ for which a given $q=q^y_{k,m,\ell,j} \in \mathcal{B}_j \cap |\gamma_\tau|$ 
is at most $M_m^{-1}$. Here $y=e$ or $w$. As before, H\"older's inequality yields 
\begin{eqnarray} 
\label{jytaa3} S_3 \leq \sum_{q \in \mathcal{B}_j} M_m^{-1} \rho(q) \leq \Big(\sum_{q \in \mathcal{B}_j} M_m^{-2} \Big)^{1/2}
\Big(\sum_{q \in \mathcal{B}_j} \rho(q)^2 \Big)^{1/2}.  
\end{eqnarray} 
We estimate the first sum from above by summing over all $q \in \mathcal{B}_j$ and applying \eqref{crazy} to have 
\begin{equation}
\label{api}
\sum_{q \in \mathcal{B}_j} M_m^{-2} \leq 2 \sum_{m=1}^{j} \sum_{k=1}^{m}\sum_{\ell=1}^{M_m} M_m^{-2} = 2 \sum_{m=1}^{j} mM_m^{-1} 
\leq \sum_{m=1}^{\infty} 2^{-m} = 1. 
\end{equation}

We have $\ell(F_j) \geq \frac{1}{2}$ by construction. Combining with \eqref{vihu}, \eqref{jytaa}, \eqref{jytaa2}, \eqref{jytaa3}, and \eqref{api}, 
yields 
\begin{eqnarray}
\nonumber \frac{1}{2}&\leq& \Big(\int_{D_j} \rho^2 \, dA \Big)^{1/2} + \Big(\sum_{\overline{U} \in \mathcal{A}_j}  \rho(\overline{U})^2\Big)^{1/2}
+\Big(\sum_{q \in \mathcal{B}_j}  \rho(q)^2\Big)^{1/2}\\ 
\label{saksi} &\leq& 3 \Big(\int_{D_j} \rho^2 \, dA + \sum_{q \in \mathcal{C}(D_j)}  \rho(q)^2\Big)^{1/2}. 
\end{eqnarray} 
Since \eqref{saksi} holds for all $\rho \in X(\Gamma)$, we conclude that 
$$
\modu(Q_{0,j_0},Q_{1,j_0};D_j) \geq \modu(\Gamma) \geq \frac{1}{36}. 
$$
The proof is complete. 


\section{Proof of Theorem \ref{mainexample2} }

\subsection{Construction of the domain}
The set $\mathcal{C}(D)$ of complementary components of $D$, which we now describe, consists of countably many segments and a Cantor set.\footnote{The size of the Cantor set is not relevant for our construction. For instance, the construction can be carried out so that $\hatc \setminus D$ has $\sigma$-finite length. } 
Let $\mathcal{W}_0=\{e\}$, $\mathcal{Y}_0=\{(e,e)\}$, and for $k=1,2,\ldots$, let 
\begin{eqnarray*}
\mathcal{W}_k &=& \{w=w_1w_2w_3 \cdots w_k: \, w_\ell \in \{0,1\} \, \text{for } 1 \leq \ell \leq k \},  \\ 
\mathcal{W}_\infty &=& \{\bar{w}=w_1w_2w_3 \cdots : \, w_\ell \in \{0,1\} \, \text{ for }\ell=1,2,\ldots\}, \quad \text{and }  \\
\mathcal{Y}_k &=& \{(w,v): \, w \in \mathcal{W}_k,\, v=v_1v_2v_3\cdots v_k, \, v_\ell \in \{0,1,2,3\}\, \text{for } 1 \leq \ell \leq k \}. 
\end{eqnarray*}
If $\bar{w}=w_1w_2\cdots \in \mathcal{W}_\infty$ and $k \in \mathbb{N}$, we denote $\bar{w}(k)=w_1\cdots w_k$. 

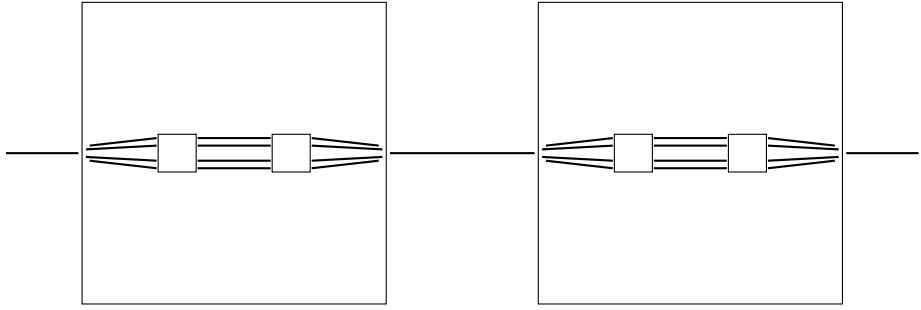
\begin{figure}
\begin{tikzpicture}
\draw (-5,-2) rectangle (-1,2);
\draw (1,-2) rectangle (5,2);
\draw[black, thick] (-6,0) -- (-5.05,0);
\draw[black, thick] (-.95,0) -- (.95,0); 
\draw[black, thick] (5.05,0) -- (6,0);  

\draw (-4,-.25) rectangle (-3.5,.25); 
\draw (-2.5,-.25) rectangle (-2,.25); 
\draw (2,-.25) rectangle (2.5,.25); 
\draw (3.5,-.25) rectangle (4,.25); 

\draw[black, thick] (-4.9,.1) -- (-4.02,.2);
\draw[black, thick] (-4.95,.05) -- (-4.02,.1); 
\draw[black, thick] (-4.95,-.05) -- (-4.02,-.1);  
\draw[black, thick] (-4.9,-.1) -- (-4.02,-.2);

\draw[black, thick] (4.9,.1) -- (4.02,.2);
\draw[black, thick] (4.95,.05) -- (4.02,.1); 
\draw[black, thick] (4.95,-.05) -- (4.02,-.1);  
\draw[black, thick] (4.9,-.1) -- (4.02,-.2);

\draw[black, thick] (-3.48,.2) -- (-2.52,.2);
\draw[black, thick] (-3.48,.1) -- (-2.52,.1); 
\draw[black, thick] (-3.48,-.1) -- (-2.52,-.1);  
\draw[black, thick] (-3.48,-.2) -- (-2.52,-.2);

\draw[black, thick] (3.48,.2) -- (2.52,.2);
\draw[black, thick] (3.48,.1) -- (2.52,.1); 
\draw[black, thick] (3.48,-.1) -- (2.52,-.1);  
\draw[black, thick] (3.48,-.2) -- (2.52,-.2);

\draw[black, thick] (4.9-6,.1) -- (4.02-6,.2);
\draw[black, thick] (4.95-6,.05) -- (4.02-6,.1); 
\draw[black, thick] (4.95-6,-.05) -- (4.02-6,-.1);  
\draw[black, thick] (4.9-6,-.1) -- (4.02-6,-.2);

\draw[black, thick] (-4.9+6,.1) -- (-4.02+6,.2);
\draw[black, thick] (-4.95+6,.05) -- (-4.02+6,.1); 
\draw[black, thick] (-4.95+6,-.05) -- (-4.02+6,-.1);  
\draw[black, thick] (-4.9+6,-.1) -- (-4.02+6,-.2);

\end{tikzpicture} 
\caption{First steps in the construction of $D$}
\end{figure}

Next, let $(R_k)$ be a sequence of positive real numbers so that $R_{k+1} < R_{k}/2$ for all $k=0,1,2, \ldots$. Moreover, 
given such a $k$ let  
\begin{equation} 
\label{sma}
\mathcal{Q}_k=\{Q_w=[x_w-R_k,x_w + R_k] \times [-R_k,R_k]: \, w \in \mathcal{W}_k\} 
\end{equation} 
be a family of disjoint, congruent closed squares in $\mathbb{C}$ with centers on the real axis so that
$$
\text{if } w \in \mathcal{W}_k \text{ and } a \in \{0,1\}, \text{ then } Q_{wa} \subset \operatorname{int}(Q_w).  
$$ 
The intersection 
\begin{equation} 
\label{oona}
K = \bigcap_{k=0}^\infty \Big( \bigcup_{w \in \mathcal{W}_k} Q_w \Big)
\end{equation} 
is a Cantor set on the real axis. It is the Cantor set part of $\mathcal{C}(D)$. Each $p=p_{\bar{w}} \in K$ is uniquely determined by the $\bar w \in \mathcal{W}_\infty$ that satisfies 
$$
\{p_{\bar w} \}= \cap_{k=1}^\infty Q_{\bar w(k)}. 
$$

We now inductively define the segments in $\mathcal{C}(D)$. The definition involves a sequence $(\epsilon_k)$ of positive real numbers converging rapidly to zero. 
We initially require that $\epsilon_k < R_{k-2}/10$. The segments are of the form 
$$
I_m(w,v)=[a_m(w,v),b_m(w,v)], \quad m=1,2,3, 
$$
where $a_m(w,v),b_m(w,v) \in \mathbb{C}$ and $(w,v) \in \mathcal{Y}_k$ for some $k=0,1,2,\ldots$. 
We denote by $\pi_1: \mathbb{C} \to \R$ the projection to the real axis. 

We first choose  
\begin{equation}
\label{huima}
[a_1,b_1]=I_1=I_1(e,e), \, [a_2,b_2]=I_2=I_2(e,e), \, [a_3,b_3]=I_3=I_3(e,e)  
\end{equation} 
of length larger than $\epsilon_1$ in $ Q_e \setminus (Q_0 \cup Q_1)$, so that  
\begin{eqnarray*}
\pi_1(a_1) &<& \pi_1(b_1)<x_0-R_1 <  \pi_1(b_1)+\epsilon_2/10, \\ 
\pi_1(a_2)-\epsilon_2/10 &<& x_0+R_1 <\pi_1(a_2), \\ 
\pi_1(a_2) &<& \pi_1(b_2) < x_1-R_1 < \pi_1(b_2)+\epsilon_2/10, \\ 
\pi_1(a_3)-\epsilon_2/10 &<& x_1+R_1<\pi_1(a_3)<\pi_1(b_3). 
\end{eqnarray*} 
We can also require the segments to be horizontal, but this is not necessary and such a requirement cannot be made below when $k\geq 1$. 

Next fix $k \geq 1$ and assume that $I_m(w',v')$ and $\epsilon_\ell$ are defined for $(w',v') \in \mathcal{Y}_{\ell}$, $0 \leq \ell \leq k-1$, so that 
\begin{equation} 
\label{gidzang}
\text{if } B_1 \in \mathcal{B}_{\ell_1} \text{ and } B_2 \in \mathcal{B}_{\ell_2}, \,B_1 \neq B_2, \text{ then }   
\overline{B}_1 \cap \overline{B}_2 = \emptyset.  
\end{equation} 
Here 
\begin{equation}
\label{holysmo}
\mathcal{B}_\ell=\{B(z,\epsilon_\ell): \, z \text{ endpoint of } I_{m}(w,v), \, (w,v) \in \mathcal{Y}_{\ell}, \, m=1,2,3 \}. 
\end{equation}
Let 
$$
I_1(w,v), \, I_2(w,v), \, I_3(w,v) \subset \operatorname{int}(Q_w) \setminus (Q_{w0} \cup Q_{w1}), \quad (w,v)=(w'\alpha,v'\beta) \in \mathcal{Y}_k, 
$$
be disjoint segments with the following properties: if we denote $a_m(w,v)=a_m$ and $b_m(w,v)=b_m$, then 
\begin{eqnarray*}
\pi_1(a_1) &<& \pi_1(b_1)<x_{w0}-R_{k+1} <  \pi_1(b_1)+\epsilon_{k+1}/10, \\ 
\pi_1(a_2)-\epsilon_{k+1}/10 &<& x_{w0}+R_{k+1} <\pi_1(a_2), \\ 
\pi_1(a_2) &<& \pi_1(b_2) < x_{w1}-R_{k+1} < \pi_1(b_2)+\epsilon_{k+1}/10, \\ 
\pi_1(a_3)-\epsilon_{k+1}/10 &<& x_{w1}+R_{k+1}<\pi_1(a_3)<\pi_1(b_3). 
\end{eqnarray*}
We also require that (See Figure \ref{positio}) if we denote 
\begin{eqnarray}
\begin{split} 
\label{timis} 
r_v = \epsilon_{k}, \,  R_v=(\epsilon_{k-1}\epsilon_{k})^{1/2} \quad \text{if } \beta=0 \text{ or } 1, \\  
r_v = (\epsilon_{k-1}\epsilon_{k})^{1/2}, \, R_v=\epsilon_{k-1} \quad \text{if } \beta=2 \text{ or } 3, 
\end{split} 
\end{eqnarray}
then for each $r_v< r < R_v$ there are arcs 
\begin{eqnarray*} 
J_1(r,w,v) \subset S(b_m(w',v'),r), \quad m=\alpha+1,  \\ 
J_3(r,w,v)\subset S(a_m(w',v'),r), \quad m=\alpha+2,  
\end{eqnarray*} 
whose relative interiors are disjoint and do not intersect any segment $I_m(\tilde w,\tilde v)$, 
$(\tilde{w},\tilde{v}) \in \mathcal{Y}_\ell$, $0 \leq \ell \leq k$, so that 
\begin{itemize} 
\item[(i)] the endpoints of $J_1(r,w,v)$ lie in $I_{\alpha+1}(w',v')$ and $I_1(w,v)$. 
\item[(ii)] the endpoints of $J_3(r,w,v)$ lie in $I_{\alpha+2}(w',v')$ and $I_3(w,v)$. 
\end{itemize}

We are now ready to define $D$; it is the domain for which 
$$
\mathcal{C}(D)=K \cup \{I_m(w,v): \, m=1,2,3, \, (w,v) \in \mathcal{Y}_{k}, \, k =0,1,2, \ldots\}, 
$$
where $K$ is the Cantor set in \eqref{oona}. 

\begin{figure} 
\begin{tikzpicture}
\draw  (6,0) circle (3cm); 
\draw (6,0) circle (2cm); 
\draw (6,0) circle (1cm);
\draw[black, very thick] (1,0) -- (6,0); 
\draw[black, very thick] (6.87,.5) -- (12,0.4); 
\draw[black, very thick] (6.87,-.5) -- (12,-0.4);  
\draw[black, very thick] (7.75,1) -- (12,0.8); 
\draw[black, very thick] (7.75,-1) -- (12,-0.8);  
\node at (3.4,0.5) {$\epsilon_{k-1}$}; 
\node at (6,2.3) {$(\epsilon_{k-1}\epsilon_k)^{1/2}$}; 
\node at (5.4,0.5) {$\epsilon_{k}$}; 
\end{tikzpicture} 
\caption{Positioning of the segments $I_m(w,v)$} \label{positio}
\end{figure}
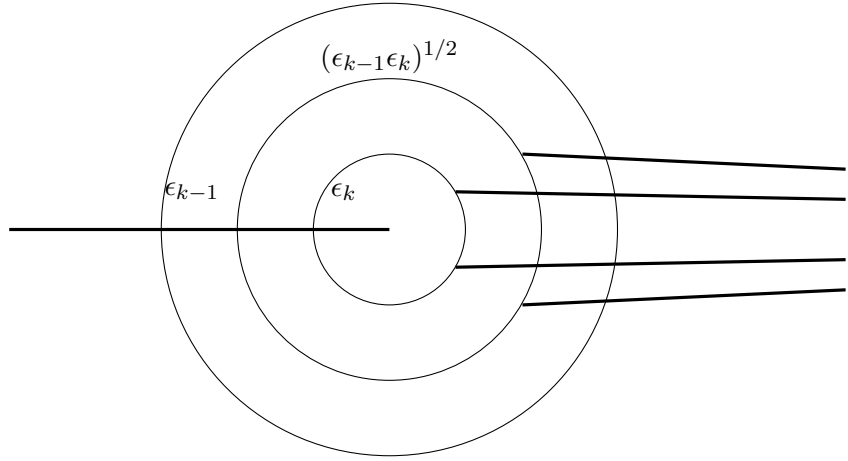


\subsection{Construction of the exhaustion} 
We now construct an exhaustion $\Phi_0=(D_j)$ of $D$. We fix $k \in \mathbb{N}$ and $(w,v) \in \mathcal{Y}_{k}$. First, let $U(w,v)$ be a Jordan domain 
so that if we denote $I(w,v)=I_1(w,v) \cup I_2(w,v) \cup I_3(w,v)$ then 
$$
I(w,v) \subset U(w,v) \subset \overline{U}(w,v) \subset \operatorname{int}(Q_w) \setminus (Q_{w0} \cup Q_{w1}). 
$$
We also require that if $B \in \cup_\ell \mathcal{B}_\ell$ where $\mathcal{B}_\ell$ is the family of balls in \eqref{holysmo}, 
then either $U(w,v) \cap B=\emptyset$ or $I(w,v) \cap B \neq \emptyset$ and 
\begin{equation} 
\label{kado}
U(w,v)\cap B \subset N_{k}(I(w,v)) \cap B. 
\end{equation} 
Here $N_k(I(w,v))$ is the set of those $x \in \mathbb{C}$ for which 
\begin{equation} 
\label{ziiri} 
\operatorname{dist}(x,I(w,v)) < \frac{\min\{\epsilon_{k+1},\operatorname{dist}(I(w,v),\mathbb{C}\setminus(D \cup I(w,v))\}}{100}. 
\end{equation} 
Next, for $m=1,2,3$ denote 
$$U_m(k+1,w,v)=U(w,v)$$ and let 
$$
U_m(j,w,v), \quad j=k+2,k+3,\ldots
$$ 
be Jordan domains so that  
\begin{eqnarray*} 
 \overline{U}_m(k+2,w,v) \cap \overline{U}_{m'}(k+2,w,v) = \emptyset \quad \text{if } m \neq m', 
\end{eqnarray*} 
and, with $N_j(I_m(w,v))$ defined as in \eqref{ziiri}, 
\begin{eqnarray*} 
I_m(w,v) \subset U_m(j,w,v) \subset \overline{U}_m(j,w,v) \subset U_m(j-1,w,v) \cap N_j(I_m(w,v)). 
\end{eqnarray*} 

We denote 
$$
\mathcal{A}_j=\{\overline{U}_m(j,w,v): \, m=1,2,3, \, (w,v) \in \mathcal{Y}_k, \, 0 \leq k \leq j-1\}, 
$$ 
and define $D_j$ by 
$$
\mathcal{C}(D_j) = \mathcal{Q}_j \cup \mathcal{A}_{j}, 
$$
where $\mathcal{Q}_j$ is the family of squares in \eqref{sma}. Then $\Phi_0=(D_j)$ is an exhaustion of $D$. 
Theorem \ref{mainexample2} follows by combining the two propositions below and choosing any $\Phi=(D_{j_n})$ so that $(f_{j_n})$ converges.

\begin{proposition} 
\label{kuja1}
There is $\delta >0$ such that if $p=p_{\bar{w}} \in K$, then 
\begin{equation}
\label{haras}
\modu(I_1,Q_{\bar{w}(j_0)};D_j) \geq \delta \quad \text{for all } j_0 \in \mathbb{N} \text{ and } j > j_0. 
\end{equation}
Here $I_1$ is the segment in \eqref{huima}. 
\end{proposition} 

Recall that, given a domain $D \subset \hatc$, $p \in \mathcal{C}(D)$, and an exhaustion $\Phi=(D_j)$ of $D$, we denote by $p_\ell$ the component in $\mathcal{C}(D_\ell)$ 
containing $p$. With this notation, $Q_{\bar{w}(j_0)}=p_{j_0}$ in \eqref{haras}. 
\begin{proposition} 
\label{combi} 
Suppose $D \subset \hatc$ is a domain with exhaustion $\Phi=(D_j)$. Fix $p \in \mathcal{C}(D)$ and a compact set $E \subset \hatc$ such that $E \cap p = \emptyset$. If 
\begin{equation}
\label{nurppa}  
\lim_{\ell \to \infty} \liminf_{j \to \infty} \modu(E,p_\ell;D_j) >0, 
\end{equation}
then $\hat{f}(p)$ is non-trivial for all $f \in \mathcal{F}_{\Phi}$. 
\end{proposition}


\subsection{Proof of Proposition \ref{kuja1} } 

Fix $p=p_{\overline{w}} \in K$, $j_0 \in \mathbb{N}$, and $j >j_0$. Let $\mathcal{V}_0=\{e\}$, and for $k=1,2,\ldots$, let 
$$
\mathcal{V}_k=\{v=v_1v_2\cdots v_k: \, v_\ell=\{0,1,2,3\} \, \text{for all } 1 \leq \ell \leq k\}, 
$$
so that $\mathcal{Y}_k=\mathcal{W}_k \times \mathcal{V}_k$. We consider the family of continua 
$$
\eta(v,t) \subset \hat{D}_j, \, v \in \mathcal{V}_{j-1}, \, 1/4<t<3/4, 
$$
defined as follows: if $v=v_1v_2\ldots v_{j-1}$, let $\eta(v,t)=A_j(v) \cup B_j(v,t)$, where 
\begin{eqnarray*}
A_j(v) &=& \cup \{\overline{U}_m(j,w,v_k) \in \mathcal{A}_j: \, m=1,2,3,\, w \in \mathcal{W}_k,\, 0 \leq k \leq j-1\}, \, \text{and }\\ 
B_j(v,t) &=& \cup \{J_m(r[t,k],w,v_k): \, m=1,3, \,  w \in \mathcal{W}_k,\, 1 \leq k \leq j-1\}. 
\end{eqnarray*}
Here $r[t,k]=R^t_{v_k}r^{1-t}_{v_k}$ and $R_{v_k},r_{v_k}$ are the radii in \eqref{timis}. 

Each $\eta(v,t)$ is a continuum joining $\overline{U}_1(j,e,e)$ and 
$\overline{U}_3(j,e,e)$ in $\hat{D}_j$. Moreover, each $\eta(v,t)$ intersects $Q_{\bar w(j_0)}$. By \eqref{kado}, we have $\eta(v,t) \setminus A_j(v) \subset D_j$. It is important to notice that the continua $\eta(v,t)$ do not intersect any of the squares in 
$\mathcal{Q}_j \subset \mathcal{C}(D_j)$. 

Let $\gamma_{v,t}$ be an arc-length parametrization of $\eta(v,t)$, and 
$$
\Gamma_j=\{\gamma_{v,t}: \, v \in \mathcal{V}_{j-1}, \, 1/4 < t < 3/4\}. 
$$
In view of the comments above, \eqref{haras} follows if we can prove a lower bound for $\modu \Gamma_j$ independent of $j$. Fix $\rho \in X(\Gamma_j)$; 
\begin{eqnarray*}
1 \leq \sum_{q \in A_j(v)}\rho(q) + \sum_{J \in B_j(v,t)} \int_J \rho \, ds. 
\end{eqnarray*}
Integrating both sides over $1/4<t<3/4$ and summing over $v \in \mathcal{V}_{j-1}$ yields 

\begin{eqnarray*}
\frac{4^{j-1}}{2}  \leq \frac{1}{2}\sum_{v \in \mathcal{V}_{j-1}} \sum_{q \in A_j(v)}\rho(q) 
+ \sum_{v \in \mathcal{V}_{j-1}}   \int_{1/4}^{3/4} \sum_{J \in B_j(v,t)} \int_{J} \rho \, ds \, dt = S_1+S_2. 
\end{eqnarray*}
 
We estimate the sums $S_1,S_2$ from above. First, changing the order of summation yields 
\begin{eqnarray*} 
2S_1 =\sum_{k=0}^{j-2} 4^{j-1-k}  \sum_{\substack{(w,v') \in \mathcal{Y}_k \\ m=1,2,3}}\rho(\overline{U}_m(j,w,v')) + \sum_{(w,v) \in \mathcal{Y}_{j-1}} \rho(\overline{U}(j,w,v)) 
= \sum_{k=0}^{j-1}S'_k. 
\end{eqnarray*} 

H\"older's inequality yields 
\begin{eqnarray*} 
S'_k \leq 4^{j-k-1} (3\cdot 2^k\cdot4^k)^{1/2} \Big( \sum_{q \in \mathcal{C}(D_j)} \rho(q)^{2}\Big)^{1/2} \leq 2^{2j-k/2-1} \Big( \sum_{q \in \mathcal{C}(D_j)} \rho(q)^{2}\Big)^{1/2}
\end{eqnarray*} 
for all $0 \leq k \leq j-1$. Thus, summing over $k$ we have 
$$
S_1 \leq 4^j \sum_{k=0}^{j-1} 2^{-k/2} \Big( \sum_{q \in \mathcal{C}(D_j)} \rho(q)^{2}\Big)^{1/2}  \leq 4^{j+1} \Big( \sum_{q \in \mathcal{C}(D_j)} \rho(q)^{2}\Big)^{1/2}. 
$$

We now estimate $S_2$. First, we denote by $\mathcal{Z}_\ell$ the set of centers $z$ in the definition of $\mathcal{B}_\ell$ in \eqref{holysmo}. Fubini's theorem and 
\eqref{gidzang} yield 
\begin{equation}
\label{ooks} 
S_2 \leq \sum_{k=1}^{j-1}4^{j-k-1} \sum_{z \in \mathcal{Z}_k} \int_{1/4}^{3/4} \int_{S(z,r[t,k])} \rho \, ds \, dt =  \sum_{k=1}^{j-1} T_k. 
\end{equation} 
We apply change of variables to the integral in \eqref{ooks} to conclude that 
\begin{equation} 
\label{sip}
T_k \leq 4^{j-k}  \Big( \log \frac{\epsilon_{k-1}}{\epsilon_k}\Big)^{-1} \sum_{z \in \mathcal{Z}_k} \int_{B(z,\epsilon_{k-1}) \setminus \overline{B}(z,\epsilon_k)} \frac{\rho(x)}{|x|}\, dA(x). 
\end{equation} 
Applying H\"older's inequality to the integral in 
\eqref{sip} yields 
$$
T_k \leq (2\pi)^{1/2} 4^{j-k}  \Big( \log \frac{\epsilon_{k-1}}{\epsilon_k}\Big)^{-1/2}  \sum_{z \in \mathcal{Z}_k}\Big(\int_{B(z,\epsilon_{k-1})} \rho(x)^2\, dA(x) \Big)^{1/2}. 
$$
Since $\operatorname{card}(Z_k) \leq 6\cdot 8^{k-1}\leq 8^k$ for all $0 \leq k \leq j-1$, we moreover have  
\begin{eqnarray*}
T_k \leq (2\pi)^{1/2} 4^j \cdot 2^{-k/2} \Big( \log \frac{\epsilon_{k-1}}{\epsilon_k}\Big)^{-1/2} \Big(\int_{D_j} \rho(x)^2 \, dA(x) \Big)^{1/2}.  
\end{eqnarray*}
Thus, if we require that $\epsilon_k \leq \epsilon_{k-1}/e$ for all $k$, we have 
\begin{eqnarray*}
S_2 \leq (2\pi)^{1/2} 4^j \sum_{k=1}^{j-1} 2^{-k/2} \Big(\int_{D_j} \rho(x)^2 \, dA(x) \Big)^{1/2} \leq  4^{j+2}  \Big(\int_{D_j} \rho(x)^2 \, dA(x) \Big)^{1/2}. 
\end{eqnarray*}
Combining the estimates yields 
\begin{eqnarray*}
4^{j-2} &\leq& 4^{j+1} \Big( \sum_{q \in \mathcal{C}(D_j)} \rho(q)^{2}\Big)^{1/2}+4^{j+2}  \Big(\int_{D_j} \rho(x)^2 \, dA(x) \Big)^{1/2}  \\ 
&\leq& 4^{j+3} \Big(\int_{D_j} \rho(x)^2 \, dA(x)+  \sum_{q \in \mathcal{C}(D_j)} \rho(q)^{2} \Big)^{1/2}. 
\end{eqnarray*} 
We conclude that $\modu(\Gamma_j) \geq 4^{-10}$. The proof is complete. 


\subsection{Proof of Proposition \ref{combi}} 
By taking a subsequence of $(D_j)$, we may assume $f_j \to f$. Suppose towards contradiction that $\hat{f}(p)$ is a point 
component. We lose no generality by assuming $\hat{f}(p)=\{0\}$.  

\begin{lemma} 
\label{umme} 
Suppose $\hat{f}(p)=\{0\}$. For every $R>0$ there are $r>0$ and $m \in \mathbb{N}$ so that if $j\geq m$ and if $q \in \mathcal{C}(D_j)$ satisfies $\hat{f}_j(q) \cap S(0,R) \neq \emptyset$, then $\hat{f}_j(q) \cap S(0,r) =\emptyset$. 
\end{lemma}

\begin{proof} 
Suppose towards contradiction that there is $R>0$ and a sequence $(q_{n_j})$, $q_{n_j} \in \mathcal{C}(D_{n_j})$, so that each $\hat{f}_{n_j}(q_{n_j})$ intersects both $S(0,R)$ 
and $S(0,2^{-j})$. By passing to a subsequence if necessary, we may assume $n_j=j$. 

For each $j \in \mathbb{N}$, fix a point $x_j \in q_j$. 
Since $\hatc \setminus D$ is compact, $(x_j)$ has a subsequence converging to $x_0 \in q_0$ for some $q_0 \in \mathcal{C}(D)$. We may assume that $x_j \to x_0$. 
It follows that if $k \in \mathbb{N}$ and if $q_0(k)$ is the element of $\mathcal{C}(D_k)$ containing $q_0$, then 
$$
q_j \subset q_0(k) \quad \text{for all } j \geq j_k. 
$$
In particular, since $\hat{f}_j(q_j)$ intersects both $S(0,R)$ and $S(0,2^{-j})$, so does $\hat{f}_j(q_0(k))$. We conclude that $\hat{f}(q_0(k))$ 
contains both the origin and a point in $S(0,R)$. But this holds for all $k$, so also $\hat{f}(q_0)$ contains both the origin and a point in $S(0,R)$. This contradicts our 
assumption, that $\hat{f}(p)=\{0\}$. The proof is complete. 
\end{proof}

We use Lemma \ref{umme} to construct a decreasing sequence $(R_n)$ of positive real numbers and an increasing sequence $j_n$ of indices as follows (compare 
to the proof of Lemma \ref{masterlemma2}): First, 
choose $R_1,j_1$ so that $\hat{f}_{j}(E) \cap B(0,2R_1)= \emptyset$ for all $j \geq j_1$. Here $E$ is the compact set in the statement of the proposition. 

Then, assuming that $R_n,j_n$ have been constructed, choose 
$R_{n+1} <R_n/2$ and $j_{n+1} \geq j_n$ such that if $q \in \mathcal{C}(D_j)$, $j \geq j_{n+1}$, and $\hat{f}_j(q) \cap S(0,R_n) \neq \emptyset$, then 
 $\hat{f}_j(q) \cap S(0,2R_{n+1}) =\emptyset$. 
 
Given $k \in \mathbb{N}$, let $N$ be the largest number for which there is $j'_N\geq k$ so that 
$\hat{f}_j(p_k) \subset B(0,R_N)$ for all $j \geq j'_N$. We may assume that $j'_N=j_N$. Then 
\begin{equation} 
\label{jenk}
\modu(\hat{f}_j(E),\hat{f}_j(p_k);f_j(D_j)) \leq \modu(S(0,2R_1),S(0,R_N);f_j(D_j)) 
\end{equation} 
for all $j \geq j_N$ (here the modulus on the left is over all paths connecting $\hat{f}_j(E)$ and $\hat{f}_j(p_k)$ in $\widehat{f_j(D_j)}$, a slight abuse of earlier terminology). 
Fix such a $j$. We construct a test function $\rho$ as follows: First, let $1 \leq n \leq N$. We denote 
$A(R,r)=B(0,R) \setminus \overline{B}(0,r)$ and define 
\begin{eqnarray*} 
\rho_n(x)=
\left\{ \begin{array}{ll} 
\frac{1}{|x| \log 2}, &  x \in D_j \cap A(2R_n,R_n) \\
\frac{\diam(x)}{R_n \log 2}, & x \in \mathcal{C}(f_j(D_j)), \, x \cap A(2R_n,R_n) \neq \emptyset, \, \diam(x) \leq \operatorname{dist}(x,0), \\ 
1, & x \in \mathcal{C}(f_j(D_j)), \, x \cap A(2R_n,R_n) \neq \emptyset, \, \diam(x) > \operatorname{dist}(x,0),
\end{array}\right. 
\end{eqnarray*}
and $\rho_n(x)=0$ otherwise. As in the proof of Lemma \ref{masterlemma2}, we have 
$$
\rho=\frac{1}{N}\sum_{n=1}^N \rho_n  \in  X(S(0,2R_1),S(0,R_N);f_j(D_j)) \quad \text{for all } j \geq j_N. 
$$ 
For each $q \in \mathcal{C}(D_j)$ there is at most one $n$ such that $\rho_n(q) \neq 0$. Moreover, for every $n$ there are at most $30$ 
elements (disks) $q \in \mathcal{C}(f_j(D_j))$ such that $q \cap A(2R_n,R_n) \neq \emptyset$ and  $\diam(q) > \operatorname{dist}(q,0)$. Thus we can estimate 
\begin{eqnarray*} 
& &\int_{f_j(D_j)} \rho_n^2 \, dA + \sum_{q \in \mathcal{C}(f_j(D_j))} \rho_n(q)^2 \leq \frac{1}{(\log 2)^2}\int_{A(2R_n,R_n)} \frac{dA}{|x|^2} \\ 
&+& \frac{\operatorname{Area}(B(0,4R_n))}{R_n^2 (\log 2)^2} +30 \leq \frac{2\pi}{\log 2}+ \frac{16\pi}{(\log 2)^2}+30 \leq 1000, 
\end{eqnarray*}
and, since we have chosen $j_k$ so that every $q \in \mathcal{C}(f_j(D_j))$ satisfies $\rho_n(q) \neq 0$ for at most one $n$, 
\begin{equation}
\label{wales}  
\int_{f_j(D_j)} \rho^2 \, dA + \sum_{q \in \mathcal{C}(f_j(D_j))} \rho(q)^2 \leq \frac{1000 N}{N^2} = \frac{1000}{N}. 
\end{equation}

Since $N \to \infty$ as $k \to \infty$, combining \eqref{wales} with \eqref{jenk} yields 
\begin{equation} 
\label{joosu}
\lim_{k \to \infty} \liminf_{j \to \infty} \modu(\hat{f}_j(E),\hat{f}_j(p_k);f_j(D_j)) =0. 
\end{equation} 
But \eqref{joosu} and the conformal invariance of the modulus contradict our assumption \eqref{nurppa}. The proof is complete. 

\vskip 10pt
\noindent
{\bf Acknowledgement.}
We thank Toni Ikonen and Dimitrios Ntalampekos for their comments and corrections.

\bibliographystyle{abbrv}
	\bibliography{ExhaustionBiblio} 

\end{document}